\documentclass[11pt]{amsart}
\usepackage{amsmath,amsfonts,amssymb,mathrsfs}
\usepackage{amssymb,mathrsfs,graphicx,enumerate,mathabx}
\usepackage{amssymb,amscd,amsthm,bbm}
\usepackage{mathtools}
\mathtoolsset{showonlyrefs}
\usepackage[retainorgcmds]{IEEEtrantools}
\usepackage{colortbl}
\usepackage{lipsum}
\usepackage{graphicx, subfigure}
\usepackage{graphicx}
\usepackage{hyperref}
\mathtoolsset{showonlyrefs}
\title[]{Global minimizers for fast diffusion versus nonlocal interactions on negatively curved manifolds}

\author[Carrillo]{Jos\'{e} A. Carrillo}
\address[Jos\'{e} A. Carrillo]{\newline Mathematical Institute, University of Oxford, Oxford OX2 6GG, UK}
\email{carrillo@maths.ox.ac.uk}

\author[Fetecau]{Razvan C. Fetecau}
\address[Razvan C. Fetecau]{\newline Department of Mathematics, Simon Fraser University, 8888 University Dr., Burnaby, BC V5A 1S6, Canada}
\email{van@math.sfu.ca}

\author[Park]{Hansol Park}
\address[Hansol Park]{\newline Department of Mathematics and Statistics, Dalhousie University, 6299 South St, Halifax, NS B3H 4R2, Canada}
\email{hansol960612@snu.ac.kr}

\newtheorem{theorem}{Theorem}[section]
\newtheorem{lemma}{Lemma}[section]

\newtheorem{proposition}{Proposition}[section]
\newtheorem{remark}{Remark}[section]
\newtheorem{example}{Example}[section]

\newcommand{\bbr}{\mathbb R}

\newcommand{\bbs}{\mathbb S}

\newcommand{\bbh}{\mathbb H}

\newcommand{\calK}{\mathcal{K}}

\newcommand{\calP}{\mathcal{P}}
\newcommand{\calPo}{\mathcal{P}_{1,o}}

\newcommand{\calW}{\mathcal{W}}
\def\d{\mathrm{d}}
\newcommand{\dV}{\mathrm{d}V} % volume measure on M
\newcommand{\dx}{\mathrm{d}V\!(x)}
\newcommand{\dy}{\mathrm{d}V\!(y)}

\newcommand{\dm}{d} % dimension

\newcommand{\p}{o}% pole of M
\newcommand{\m}{q}% exponent for nonlinear diffusion
\newcommand{\h}{h}%potential defined by the intrinsic distance square
\newcommand{\dist}{d}%intrinsic distance
\newcommand{\secc}{\mathcal{K}}%sectional curvature
\newcommand{\ind}{\chi}%indicator function

\makeatletter
\@namedef{subjclassname@2020}{\textup{2020} Mathematics Subject Classification}
\makeatother
\begin{document}

\date{\today}


\begin{abstract}
We investigate the existence of ground states for a free energy functional on Cartan-Hadamard manifolds. The energy, which consists of an entropy and an interaction term, is associated to a macroscopic aggregation model that includes nonlinear diffusion and nonlocal interactions. We consider specifically the regime of fast diffusion, and establish necessary and sufficient conditions on the behaviour of the interaction potential for global energy minimizers to exist. We first consider the case of manifolds with constant bounds of sectional curvatures, then extend the results to manifolds with general curvature bounds. To establish our results we derive several new Carlson-Levin type inequalities for Cartan-Hadamard manifolds.
\end{abstract}

\maketitle \centerline{\date}
%\tableofcontents

\section{Introduction}
In this work we consider a generic Cartan-Hadamard manifold $M$, and investigate the existence of global minimizers of the following free energy functional:
\begin{align}
\label{eqn:E-gen}
E[\mu]=\frac{1}{\m-1}\int_M\rho(x)^\m \dx+\frac{1}{2}\iint_{M \times M} W(x, y) \d\mu(x)\d \mu(y),
\end{align}
where $0<\m<1$, $W: M\times M \to \bbr$ is an attractive interaction potential, and $\dV$ denotes the Riemannian volume measure on $M$. The energy is defined on the space $\calP(M)$ of probability measures on $M$, and $\rho$ represents the absolutely continuous part of $\mu\in\calP(M)$ with respect to $\dV$. The first component of the energy is entropy, which favours spreading, and the second is an interaction energy, which in our case promotes aggregation, at least for long distances.

The free energy \eqref{eqn:E-gen} relates to the following evolution equation on $L^1(M) \cap \calP(M)$:
\begin{align}\label{eqn:model}
\partial_t\rho(x)- \nabla_M \cdot(\rho(x)\nabla_M W*\rho(x))=\Delta \rho^\m(x),
\end{align}
where
\[
W*\rho(x)=\int_M W(x, y)  \d \rho(y),
\]
and $\nabla_M \cdot$ and $\nabla_M $ represent the manifold divergence and gradient, respectively. Equation \eqref{eqn:model} can in fact be formally expressed as the gradient flow of the energy $E$ on a suitable Wasserstein space \cite{AGS2005}.

Energies in the form \eqref{eqn:E-gen} and the associated evolution equation \eqref{eqn:model}, known as the aggregation-diffusion equation, have a wide range of applications in mathematical biology. For instance, nonlinear diffusion versions of the classical Keller-Segel model for chemotaxis or swarming models (see \cite{CarrilloCraigYao2019} and the references therein) have been well studied. Such models also appear as paradigms of social interactions where they can lead to segregation of individuals into clusters of opinions, see \cite{GPY17,GPY19} for instance. For specific manifold setups, we note the work in \cite{fatkullin2005critical}, where the authors used the Onsager model to analyze phase transitions phenomena on the two dimensional sphere. Finally, consensus of large number of individuals moving under constraints, modelled by aggregation-diffusion equations on manifolds, can also find applications in control engineering or swarming models with body-attitude coordination, see \cite{DDFM20}.

The range $0<\m<1$ of the nonlinear diffusion exponent is referred to in the literature as {\em fast} diffusion, while $\m=1$ and $\m>1$ correspond to {\em linear} and {\em slow} diffusion, respectively; note that for linear diffusion, the entropy term in \eqref{eqn:E-gen} is replaced by $\int_M \rho(x) \log \rho(x) \dx$.  The fast diffusion is faster than the linear/slow diffusion for small values of the density while it is actually slower for large densities (concentrations). Consequently, fast diffusion comes with an enhanced spreading at infinity, but at the same time, by competing with the nonlocal attractive interactions in our model, it may not be able to prevent Dirac delta aggregations from forming \cite{carrillo2019reverse}. This is in contrast with the linear and slow diffusion cases. 

The free energy \eqref{eqn:E-gen} and the aggregation-diffusion equation \eqref{eqn:model} have been extensively studied in the Euclidean setup $M = \bbr^\dm$. A recent review on the subject can be found in \cite{CarrilloCraigYao2019}. In \cite{CalvezCarrilloHoffmann2017,CaDePa2019,carrillo2019reverse,CaDeFrLe2022,CaHoMaVo2018} the authors considered an interaction potential of Riesz type, i.e., $W(x, y)=\frac{1}{\beta}\|x-y\|^\beta$, where $\|\cdot\|$ denotes the Euclidean norm and established existence or non-existence of global energy minimizers for various ranges of $\beta$ and $\m$. Uniqueness and qualitative properties (e.g., monotonicity and radial symmetry) of energy minimizers or steady states of \eqref{eqn:model} have been studied in \cite{Bedrossian11,CaHiVoYa2019,DelgadinoXukaiYao2022,Kaib17}. Well-posedness, long time behaviour and properties of solutions to the evolution equation \eqref{eqn:model} were addressed in \cite{CarrilloCraigYao2019, CaHiVoYa2019, fernandez2023partial}. In particular, the authors of \cite{fernandez2023partial, CGV22} studied the partial mass concentration in time for the fast diffusion, both for confinement and interaction potentials. In our current research we deal with the first step in understanding fast diffusion phenomena on general manifolds. Hence, we focus on finding conditions that ensure the existence of global minimizers for the free energy \eqref{eqn:E-gen}, without discussing the possible formation of concentrations, which are still quite intricate and open even in the Euclidean case \cite{CaDeFrLe2022}.

In spite of extensive studies of the free energy \eqref{eqn:E-gen} and the evolution equation \eqref{eqn:model} in the Euclidean space, there has been very little research done in the manifold setup. In this paper we consider the free energy \eqref{eqn:E-gen} on general Cartan-Hadamard manifolds. Important examples of Cartan-Hadamard manifolds are the hyperbolic space and the space of symmetric positive definite matrices. In general, any non-compact symmetric space is Cartan-Hadamard \cite{Helgason2001}. We consider interaction potentials $W(x,y)$ that depend only on the geodesic distance $\dist(x,y)$ between the points $x$ and $y$. This assumption, which makes the model completely {\em intrinsic} to the given manifold, was used in recent works on the interaction model with no diffusion \cite{FeHaPa2021,FePa2023b,FePa2021,FeZh2019}. With this setup, we address and answer the following problem. Assume that the sectional curvatures at a generic point $x$ are bounded below and above, respectively, by two negative functions $-c_m(r_x)$ and $-c_M(r_x)$, where $r_x$ denotes the distance from $x$ to a fixed (but arbitrary) pole on $M$. A particular case is when $c_m(\cdot)$ and $c_M(\cdot)$ are constant functions, but in general $c_m(\cdot)$ and $c_M(\cdot)$ are allowed to grow at any rate at infinity, e.g., they can grow algebraically or exponentially.  We are interested to account for the variable (and possibly unbounded) curvature of the manifold, and establish necessary and sufficient conditions on the behaviour of the interaction potential at zero and infinity, for global minimizers of the free energy to exist.

The problem posed above was studied in \cite{CaDePa2019} for $M = \bbr^\dm$ and $\m>0$ (all the three diffusion regimes), in \cite{CaFePa2024a} for $M=\bbh^\dm$ -- the hyperbolic space -- and $\m>1$ (slow diffusion), and in \cite{FePa2024b,FePa2024a} for general Cartan-Hadamard manifolds and $\m=1$ (linear diffusion). In particular, \cite{CaDePa2019} provided necessary, but not shown to be also sufficient, conditions on the interaction potential for existence of ground states in the fast diffusion regime. In the case of linear diffusion, \cite{FePa2024a,FePa2024b} demonstrate the role of the manifold's curvature with regards to spreading and blow-up. For example, spreading by linear diffusion in $\bbr^\dm$ is prevented provided the interaction potential grows at least logarithmically at infinity; the condition is in fact sharp \cite{CaDePa2019}. On the other hand, to contain the linear diffusion on general Cartan-Hadamard manifolds with bounded sectional curvatures, superlinear growth of the interaction potential at infinity is needed \cite{FePa2024a}. Therefore, as a manifestation of the curvature, a stronger growth of the attractive potential is needed to contain diffusion on manifolds.

The present work is the first to consider fast diffusion and nonlocal interactions on general manifolds of negative sectional curvatures. We quantify, in terms of the upper and lower bounds of the sectional curvatures, the conditions on $W$ that lead to existence or nonexistence of global energy minimizers. The main tools we use in our analysis are various comparison results in differential geometry (Rauch theorem, volume bounds, bounds for the Jacobian of the exponential map), as well as several new generalized Carlson-Levin inequalities on Cartan-Hadamard manifolds, that we derived for our purpose and have an interest in their own (Lemma \ref{mainineq} and \ref{maininequn}). Compared to their analogue from the Euclidean space \cite[Lemma 5]{carrillo2019reverse}, our results contain additional terms that reflect the volume growth of geodesic balls on Cartan-Hadamard manifolds.

Finally, we note that there has been strong interest in recent years on PDEs that involve nonlinear diffusion on Cartan-Hadamard manifolds. Issues that have been addressed for instance are the well-posedness and long-time behaviour of solutions to the porous medium equation on negatively curved manifolds \cite{GrilloMuratoriVazquez2017, GrilloMuratoriVazquez2019}, as well as to PDEs that include both nonlinear diffusion and local reaction terms \cite{ GrilloMeglioliPunzo2021a, GrilloMeglioliPunzo2021b}. Due to the nonlocal nature of our interaction term, the methods developed in our work are very different from the approaches taken in these papers.

The summary of the paper is as follows. In Section \ref{sect:prelim} we present the assumptions, the variational setup, and some necessary background in Riemannian geometry. In Section \ref{sect:nonexist} we establish necessary conditions on the interaction potential for global energy minimizers to exist (Theorem \ref{thm:nonexist}). In Section \ref{sect:C-Lineq} we derive a generalized Carlson-Levin inequality for Cartan-Hadamard manifolds with sectional curvatures bounded from below. This inequality is used in Section \ref{sect:existence}  to establish sufficient conditions on the growth of the interaction potential at infinity, for a ground state to exist (Theorem \ref{thm:existence}). In Section \ref{sect:unbounded} we generalize the previous existence/nonexistence results, as well the Carlson-Levin inequality, to Cartan-Hadamard manifolds with unbounded sectional curvatures.

 %%%%%

\section{Preliminaries and background}
\label{sect:prelim}
\setcounter{equation}{0}

\subsection{Assumptions and notations}
\label{subsect:notations}
We make the following assumptions on the manifold $M$ and the interaction potential $W$.
\smallskip

\noindent(\textbf{M}) $M$ is an $\dm$-dimensional Cartan-Hadamard manifold, i.e., $M$ is complete, simply connected, and has everywhere non-positive sectional curvature. In particular, $M$ has no conjugate points, it is diffeomorphic to $\bbr^\dm$ and the exponential map at any point is a global diffeomorphism. 
\smallskip

\noindent(\textbf{W}) The interaction potential $W:M\times M\to\bbr$ has the form
\[
W(x, y)=\h(\dist(x, y)),\qquad \text{ for all }x, y\in M,
\]
where $\h:[0, \infty)\to \bbr $ is lower semi-continuous and non-decreasing, and $\dist$ denotes the intrinsic distance on $M$. As $h$ is non-decreasing, $h$ models purely {\em attractive} interactions. 
\smallskip

{\em Notations.} For $x \in M$, and $\sigma$ a two-dimensional subspace of $T_x M$, $\calK(x;\sigma)$ denotes the sectional curvature of $\sigma$ at $x$. We also denote by $\dV$ the Riemannian volume measure on $M$ and by $\calP_{ac}(M) \subset \calP(M)$ the space of probability measures on $M$ that are absolutely continuous with respect to $\dV$.

A point on a manifold is called a pole if the exponential map at that point is a global diffeomorphism. On Cartan-Hadamard manifolds all points have this property. Throughout the paper we fix a generic point $\p$ in $M$ which we will be referring to as the {\em pole} of the manifold.  We make the notation
\[
r_x := \dist(o,x), \qquad \text{ for all } x \in M.
\]
We denote by $B_r(\p)$ the open geodesic ball centred at $\p$ of radius $r$, and by $|B_r(\p)|$ its volume. Also, $\bbs^{\dm-1}$ denotes the unit sphere in $\bbr^\dm$ and $\omega(\dm)$ its volume.
\smallskip

{\em Definition of the energy.} Any probability measure $\mu \in \calP(M)$ can be written uniquely, by Lebesgue's decomposition theorem, as $\mu = \rho \, \dV + \mu^s$, where the first term is the absolutely continuous part of $\mu$ w.r.t. the volume measure $\dV$ and the second term is the singular part. Given the assumption on the interaction potential, the free energy $E: \calP(M) \to [-\infty,\infty]$ is given by
\begin{equation}
\label{eqn:energy}
E[\mu]=\frac{1}{\m-1}\int_M\rho(x)^\m\dx+\frac{1}{2}\iint_{M \times M} h(\dist(x, y))\d\mu(x)\d \mu(y)
\end{equation}
The first term in \eqref{eqn:energy} represent the entropy and the second is the interaction energy.

As explained in \cite{CaDePa2019,carrillo2019reverse}, the energy in \eqref{eqn:energy} can first be defined for functions in $C_c^\infty(M)\cap\calP(M)$, and then extended to the entire $\calP(M)$ by
\[
E[\mu]=\inf_{\substack{\{\mu_n\}\subset C_c^\infty(M)\cap\mathcal{P}(M)\\\text{s.t.}\mu_n\rightharpoonup \mu}}\liminf_{n\to\infty}E[\mu_n],
\]
where $\mu_n \rightharpoonup \mu$ denotes weak convergence as measures, i.e.,
\[
\int_M f(x) \mu_n(x) \dx \to \int_M f(x) \mu(x) \dx, \qquad \text{ as } n\to  \infty,
\]
for all bounded and continuous functions $f:M\to \bbr$.

{\em Wasserstein distance.} Denote by $\calP_1(M)$ the set of probability measures on $M$ with finite first moment, i.e.,
\[
\calP_1(M) = \left \{\mu \in \calP(M) : \int_M \dist(x,x_0) \d \mu(x) < \infty \right \},
\]
for some fixed (but arbitrary) point $x_0 \in M$. For $\rho,\sigma \in \calP(M)$, the \emph{intrinsic} $1$-Wasserstein distance is defined as
\begin{equation*}
	\calW_1(\rho,\sigma) = \inf_{\gamma \in \Pi(\rho,\sigma)} \iint_{M \times M} \dist(x,y) \d\gamma(x,y),
\end{equation*}
where $\Pi(\rho,\sigma) \subset \calP(M\times M)$ is the set of transport plans between $\rho$ and $\sigma$, i.e., the set of elements in $\calP(M\times M)$ with first and second marginals $\rho$ and $\sigma$, respectively. The space $(\calP_1(M),\calW_1)$ is a metric space \cite{AGS2005}. 
\smallskip

{\em Admissible set.} Similar to the approach in the Euclidean space \cite{CaHiVoYa2019,CaHoMaVo2018,DelgadinoXukaiYao2022} or for general manifolds and linear diffusion \cite{FePa2024a,FePa2024b}, we fix the Riemannian centre of mass of the admissible minimizers. For a fixed pole $\p \in M$, we denote
\begin{equation}
\label{eqn:calP-1o}
\calPo(M)=\left\{\mu\in \mathcal{P}_1(M): \int_M \log_\p x \, \d \mu(x)=0\right\},
\end{equation}
where $\log_\p$ denotes the Riemannian logarithm map at $\p$ (i.e., the inverse of the Riemannian exponential map $\exp_\p$) \cite{doCarmo1992}. The integral condition in \eqref{eqn:calP-1o} fixes the centre of mass of $\rho$ at the pole $\p$. It is in this set that we will establish existence of global energy minimizers.

%\[
%\mathcal{P}_\p(M):=\left\{\rho\in\mathcal{P}_{ac}(M) \cap \mathcal{P}_1(M): \int_M \log_\p x\rho(x)\dx=0\right\},
%\]

%%%%%

\subsection{Comparison results in differential geometry}
\label{subsect:comparison}
We list here several comparison results in differential geometry, that will be used in our analysis.

\begin{theorem}[Rauch comparison theorem \cite{cheeger1975comparison}]\label{RCT}
Let $M$ and $\tilde{M}$ be Riemannian manifolds with $\mathrm{dim} (\tilde{M})\geq \mathrm{dim}(M$), and suppose that for all $p\in M$, $\tilde{p}\in \tilde{M}$, and $\sigma\subset T_pM$, $\tilde{\sigma}\subset T_{\tilde{p}}\tilde{M}$, the sectional curvatures $\secc$ and $\tilde{\secc}$ of $M$ and $\tilde{M}$, respectively, satisfy
\[
\tilde{\secc} (\tilde{p};\tilde{\sigma})\geq \secc(p;\sigma).
\]
Let $p\in M$, $\tilde{p}\in\tilde{M}$ and fix a linear isometry $i:T_pM\to T_{\tilde{p}}\tilde{M}$. Let $r>0$ be such that the restriction ${\exp_p}_{|B_r(0)}$ is a diffeomorphism and ${\exp_{\tilde{p}}}_{|B_r(0)}$ is non-singular. Let $c:[0, a]\to\exp_p(B_r(0))\subset M$ be any differentiable curve and define  $\tilde{c}:[0, a]\to\exp_{\tilde{p}}(B_r(0))\subset\tilde{M}$ by
\[
\tilde{c}(s)=\exp_{\tilde{p}}\circ i\circ\exp^{-1}_p(c(s)),\qquad s\in[0, a].
\]  
Then the length of $c$ is greater or equal than the length of $\tilde{c}$.
\end{theorem}

For manifolds with bounded sectional curvatures, the Jacobian of the exponential map and the volumes of geodesic balls can be bounded below and above as follows. The results can be found in \cite[Section III.4]{Chavel2006}. 
\begin{theorem}[Jacobian and volume bounds \cite{Chavel2006}] 
\label{lemma:Chavel-thms}
Suppose the sectional curvatures of $M$ satisfy
\begin{equation}
-c_m\leq\calK(x;\sigma)\leq -c_M<0,
\label{eqn:const-bounds}
\end{equation}
for all $x \in M$ and all two-dimensional subspaces $\sigma \subset T_x M$, where $c_m$ and $c_M$ are positive constants. Let $\p$ be an arbitrary pole on $M$. Then, the Jacobian of the exponential map at $\p$ can be bounded as
\begin{equation}
\label{eqn:J-bounds}
\left(
\frac{\sinh(\sqrt{c_M}\|u\|)}{\sqrt{c_M}\|u\|}
\right)^{\dm-1}\leq |J(\exp_\p)(u)|\leq\left(
\frac{\sinh(\sqrt{c_m}\|u\|)}{\sqrt{c_m}\|u\|}
\right)^{\dm-1},
\end{equation}
for all $u \in T_\p M$.

Also, the volumes of geodesic balls are bounded by
\begin{equation}
\label{eqn:vol-bounds}
 \dm w(\dm)\int_0^r \left(\frac{\sinh(\sqrt{c_M}t)}{\sqrt{c_M}}\right)^{\dm-1} \d t \leq |B_r(\p)|\leq \dm w(\dm) \int_0 ^r \left(\frac{\sinh(\sqrt{c_m} t)}{\sqrt{c_m}}\right)^{\dm-1} \d t.
\end{equation}
\end{theorem}

\begin{remark}
\normalfont
The two bounds in \eqref{eqn:vol-bounds} represent the volume of the ball of radius $r$ in the hyperbolic space $\bbh^\dm$ of constant curvatures $-c_M$, and $-c_m$, respectively.
\end{remark}

The comparison results in Theorem \ref{lemma:Chavel-thms} can be generalized to manifolds of unbounded curvature.  The results are the following (see \cite[Theorem 2.2 and Corollary 2.1]{FePa2024b}, also \cite[Theorem 1.3]{alias2019maximum}).

\begin{theorem}(Generalized Jacobian and volume bounds \cite{FePa2024b})
\label{lemma:Chavel-thms-gen}
Suppose the sectional curvatures of  $M$ satisfy
\begin{equation}
-c_m(r_x)\leq \mathcal{K}(x;\sigma)\leq -c_M(r_x) < 0, 
\label{eqn:var-bounds}
\end{equation}
for all $x \in M$ and all two-dimensional subspaces $\sigma \subset T_x M$, where $c_m(\cdot)$ and $c_M(\cdot)$ are positive continuous functions of the distance from a fixed pole $\p$. Denote by $\psi_m$ and $\psi_M$ the solutions of the following initial-value-problems (IVP's):
\begin{equation}
\begin{cases}
\psi_m''(\theta)=c_m(\theta)\psi_m(\theta),\quad\forall \theta>0, \\[5pt]
\psi_m(0)=0,\quad \psi_m'(0)=1,
\end{cases}\quad \text{ and } \quad 
\begin{cases}
\psi_M''(\theta)=c_M(\theta)\psi_M(\theta),\quad\forall \theta>0, \\[5pt]
\psi_M(0)=0,\quad \psi_M'(0)=1.
\end{cases}
\label{eqn:psimM}
\end{equation}

Then, the Jacobian of the exponential map at $\p$ can be bounded as
\[
\left( \frac{\psi_M(\|u\|)}{\|u\|} \right)^{\dm-1}\leq |J(\exp_\p)(u)| \leq
\left( \frac{\psi_m(\|u\|)}{\|u\|} \right)^{\dm-1},
\]
for all $u \in T_\p M$. Also, the volumes of geodesic balls in $M$ can be bounded by
\[
 \dm w(\dm)\int_0^r \psi_M(t) ^{\dm-1} \d t \leq |B_r(\p)|\leq \dm w(\dm) \int_0 ^r \psi_m(t)^{\dm-1} \d t.
\]
\end{theorem}

\begin{remark} 
\label{rmk:c-consts}
\normalfont
When $c_m(\cdot) \equiv c_m$ and $c_M(\cdot) \equiv c_M$ are constant functions, the IVPs in \eqref{eqn:psimM} have the explicit solutions
 \[
\psi_m(\theta)=\frac{\sinh(\sqrt{c_m}\theta)}{\sqrt{c_m}},\quad \text{ and } \quad \psi_M(\theta)=\frac{\sinh(\sqrt{c_M}\theta)}{\sqrt{c_M}}.
\]
In this case, Theorem \ref{lemma:Chavel-thms-gen} reduces to Theorem \ref{lemma:Chavel-thms}. 

We also note that while in general, one cannot find explicit solutions to the IVP's in \eqref{eqn:psimM}, these are linear problems and the global well-posedness of their solutions is guaranteed by standard results in ODE theory. In addition, by straightforward arguments, in can be shown that $\psi_m$ and $\psi_M$ satisfy
\begin{equation}
\label{eqn:psi-ineq-mM}
 \psi_m(\theta) \geq 0, \; \psi_M(\theta) \geq 0 \qquad \text{ and } \qquad \psi_m'(\theta)\geq 1, \; \psi_M'(\theta)\geq 1, \quad \text{ for all } \theta \geq 0.
\end{equation}
\end{remark}

\begin{remark}
\normalfont
We note that for the upper bounds in Theorems \ref{lemma:Chavel-thms} and \ref{lemma:Chavel-thms-gen}, a weaker condition on curvature can be assumed, namely that the Ricci curvatures are greater than or equal to $-(\dm-1) c_m$ (see Theorems III.4.3 and III.4.4 in \cite{Chavel2006} for the constant curvature bounds case). We do not consider this weaker assumption here. 
\end{remark}

The following lemma, listed verbatim from \cite{FePa2024b} will be needed in our analysis. Here, $\overline{D}f$ denotes the upper derivatives of a real-valued function $f$, defined by
\[
\overline{D}f(x)=\limsup_{h\to0}\frac{f(x+h)-f(x)}{h}.
\]

\begin{lemma}
\label{estpsi}
Let $c:[0,\infty) \to (0,\infty)$ be a positive and continuous function, and consider the solution $\psi$ of the IVP:
\begin{equation}
\label{eqn:ivp-psi}
\begin{cases}
\psi''(\theta)=c (\theta)\psi (\theta),\qquad\forall \theta>0, \\[5pt]
\psi(0)=0,\quad \psi'(0)=1.
\end{cases}
\end{equation}

\noindent (Upper bound) Assume $c(\cdot)$ is non-decreasing. Then, 
\begin{equation}
\label{estpsires-relaxed}
\psi(\theta)\leq
\theta\exp\left(
\int_{0}^\theta\sqrt{c(t)}\, \d t \right), \qquad \forall  \theta \geq0.
\end{equation}

\noindent (Lower bound) Assume $c(\cdot)$ satisfies
\begin{equation}
\label{eqn:c32}
\lim_{\theta\to\infty}\frac{\overline{D}c(\theta)}{c(\theta)^{3/2}}=0.
\end{equation}

Then, for any $0<\epsilon <1$, there exists $\theta_0>0$ such that
\begin{equation}
\label{estpsires}
\psi(\theta_0)\exp\left(
(1-\epsilon)\int_{\theta_0}^\theta\sqrt{c(t)}\, \d t
\right)\leq \psi(\theta), \qquad \forall \theta \geq \theta_0.
\end{equation}
\end{lemma}
\begin{proof}
See  \cite[Lemma 2.1]{FePa2024b}.
\end{proof}

%%%%%%%%%%
 
\section{Nonexistence of a global minimizer} 
\label{sect:nonexist}
\setcounter{equation}{0}

In this section, we investigate necessary conditions on the interaction potential for global minimizers of the energy \eqref{eqn:energy} to exist. We assume here that the sectional curvatures of $M$ satisfy
\begin{equation}
\label{eqn:K-upperb}
\mathcal{K}(x;\sigma)\leq -c_M< 0, 
\end{equation}
for all $x \in M$ and all two-dimensional subspaces $\sigma \subset T_x M$, where $c_M$ is a positive constant.

Fix a pole $\p \in M$ and consider the following family of probability density functions $\rho_R$ that depend on $R>0$:
\begin{equation}
\label{eqn:rhoR}
\rho_{R}(x)=\begin{cases}
\displaystyle\frac{1}{|B_R(\p)|},&\quad\text{when }x\in B_R(\p),\\[10pt]
0,&\quad\text{otherwise}.
\end{cases}
\end{equation}
We will estimate $E[\rho_R]$ in two limits: $R \to 0+$ (blow-up) and $R \to \infty$ (spreading). 

By the volume comparison in Theorem \ref{lemma:Chavel-thms}, the volume $|B_R(\p)|$ can be estimated as 
\begin{equation}
\label{eqn:vol-lowb}
|B_R(\p)| \geq \dm w(\dm) \int_0^R \left(\frac{\sinh(\sqrt{c_M}\theta)}{\sqrt{c_M}}\right)^{\dm-1}\d \theta.
\end{equation}
Note that when $c_M=0$, \eqref{eqn:vol-lowb} reduces to
\begin{equation}
\label{eqn:vol-lowb-0}
|B_R(\p)| \geq \dm w(\dm) \int_0^R \theta^{\dm-1}\d \theta = w(\dm) R^\dm,
\end{equation}
which states that $|B_R(\p)|$ is bounded below by the volume of the ball of radius $R$ in $\bbr^d$. The entropy term for $\rho_R$ is given explicitly by
\begin{equation}
\label{eqn:entropy-rhoR}
\frac{1}{\m-1}\int_M \rho_R(x)^{\m}dx=-\frac{|B_R(\p)|^{1-\m}}{1-\m}.
\end{equation}

Also, by the monotonicity of $h$, the interaction term can be estimated as 
\begin{equation}
\label{eqn:intenergy-rhoR}
\begin{aligned}
\frac{1}{2}\iint_{M \times M} h(\dist(x, y))\rho_R(x)\rho_R(y)\d x\d y &\leq \frac{1}{2}\iint_{M \times M} h(2R)\rho_R(x)\rho_R(y)\d x \d y \\
&=\frac{h(2R)}{2}.
\end{aligned}
\end{equation}
Using the two equations above in \eqref{eqn:energy}, we then have
\begin{align}\label{Eupperbound}
E[\rho_R]\leq -\frac{|B_R(\p)|^{1-\m}}{1-\m}+\frac{h(2R)}{2}.
\end{align}

\noindent {\em Behaviour at zero.} Using \eqref{Eupperbound} and $0<\m<1$, we get
\[
\lim_{R\to0+}E[\rho_R]\leq \lim_{R\to0+}\left(-\frac{|B_R(\p)|^{1-\m}}{1-\m}+\frac{h(2R)}{2}\right)=\frac{1}{2}\lim_{R\to0+}h(2R).
\]
This implies that, if $h$ is singular at $0+$ (i.e., $\lim_{\theta \to 0^+} h(\theta)= -\infty$), then the energy is unbounded from below and there is no global energy minimizer.
\smallskip

\noindent {\em Behaviour at infinity.} By combining \eqref{eqn:vol-lowb} and \eqref{Eupperbound}, we get
\[
\lim_{R\to\infty}E[\rho_R]\leq \lim_{R\to\infty}\left(-\frac{1}{1-\m}\left(\dm w(\dm) \int_0^R\left(\frac{\sinh(\sqrt{c_M}\theta)}{\sqrt{c_M}}\right)^{\dm-1}\d\theta\right)^{1-\m}+\frac{h(2R)}{2}\right).
\] 
%It implies that provided
%\[
%\lim_{R\to\infty}\left(-\frac{1}{1-m}\left(\int_0^R\left(\frac{\sinh(\sqrt{c_M}\theta)}{\sqrt{c_M}}\right)^{\dm-1}\d\theta\right)^{1-m}+\frac{h(2R)}{2}\right)=-\infty,
%\]
Similar to the behaviour at zero, provided the limit in the right-hand-side above is $-\infty$, $E$ is not bounded from below and hence, it has no global minimizer. We quantify this behaviour in terms of $h$, as following. 

By L'H\^opital's rule, we have
\begin{align*}
\lim_{R\to\infty}\frac{\int_0^R\left(\frac{\sinh(\sqrt{c_M}\theta)}{\sqrt{c_M}}\right)^{\dm-1}\d\theta}{\exp(\sqrt{c_M}(\dm-1)R)}
&=\lim_{R\to\infty}\frac{
\left(\frac{\sinh(\sqrt{c_M}R)}{\sqrt{c_M}}\right)^{\dm-1}
}{\sqrt{c_M}(\dm-1)\exp(\sqrt{c_M}(\dm-1)R)}\\
%&=\lim_{R\to\infty}\frac{\left(e^{\sqrt{c_M}R}-e^{-\sqrt{c_M}R}\right)^{\dm-1}
%}{2^{\dm-1}\sqrt{c_M}^\dm(\dm-1)\exp(\sqrt{c_M}(\dm-1)R)}\\
&=\frac{1}{2^{\dm-1}\sqrt{c_M}^{\dm}(\dm-1)}.
\end{align*}
%we find that $h$ must growth at least exponentially at infinity for ground states to exist. 
If $h$ satisfies
\[
\lim_{R\to\infty}\frac{h(R)} {\exp\left(\frac{\sqrt{c_M}(\dm-1)(1-\m)}{2}R\right)}=0\quad\Longleftrightarrow\quad\lim_{R\to\infty}\frac{h(2R)}{\exp(\sqrt{c_M}(\dm-1)(1-\m)R)}=0,
\]
then we have
\begin{align*}
&\lim_{R\to\infty}\left(-\frac{1}{1-\m}\left(\dm w(\dm) \int_0^R\left(\frac{\sinh(\sqrt{c_M}\theta)}{\sqrt{c_M}}\right)^{\dm-1}\d\theta\right)^{1-\m}+\frac{h(2R)}{2}\right)\\
&\quad = \lim_{R\to\infty}\exp(\sqrt{c_M}(\dm-1)(1-\m)R)  \\ 
& \qquad \qquad \times \left(-\frac{1}{1-\m}\left( \frac{\dm w(\dm) \int_0^R\left(\frac{\sinh(\sqrt{c_M}\theta)}{\sqrt{c_M}}\right)^{\dm-1}\d\theta}{\exp(\sqrt{c_M}(\dm-1)R)}\right)^{1-\m}+\frac{h(2R)}{2\exp(\sqrt{c_M}(\dm-1)(1-\m)R)}\right)\\
%&\quad = \lim_{R\to\infty}\exp(\sqrt{c_M}(\dm-1)(1-m)R)\\
%&\qquad \qquad \times \lim_{R\to\infty}\left(-\frac{1}{1-m}\left(\frac{\int_0^R\left(\frac{\sinh(\sqrt{c_M}\theta)}{\sqrt{c_M}}\right)^{\dm-1}\d\theta}{\exp(\sqrt{c_M}(\dm-1)R)}\right)^{1-m}+\frac{h(2R)}{2\exp(\sqrt{c_M}(\dm-1)(1-m)R)}\right)\\
&\quad = \infty\times\left(-\frac{1}{1-\m} \left(\frac{\dm w(\dm)}{2^{\dm-1}\sqrt{c_M}^{\dm}(\dm-1)}\right)^{1-\m}\right)=-\infty.
\end{align*}

The two cases prove the following theorem.
\begin{theorem}[Nonexistence by blow-up or spreading]
\label{thm:nonexist}
Let $M$ be a $\dm$-dimensional Cartan-Hadamard manifold with sectional curvatures that satisfy \eqref{eqn:K-upperb} for a positive constant $c_M$.  Let $0<\m<1$ and assume that $h:[0,\infty) \to [-\infty,\infty)$ is a non-decreasing lower semi-continuous function which satisfies either
\[
\lim_{\theta\to0+}h(\theta)=-\infty,
\]
or
\begin{align}\label{non-exist-spread}
\lim_{\theta \to\infty}\frac{h(\theta )}{\exp\left(\frac{\sqrt{c_M}(\dm-1)(1-\m)}{2}\, \theta \right)}=0.
\end{align}
Then, the energy functional \eqref{eqn:energy} is unbounded from below, and hence has no global minimizer in $\calP(M)$.
\end{theorem}

\begin{remark}[Comparison with linear and slow diffusion]
\label{rmk:compare}
 The necessary conditions in Theorem \ref{thm:nonexist} say that no singularity at origin in the interaction potential is allowed, or else nonexistence by blow-up occurs. Also, the potential needs to grow exponentially fast at infinity, in order to contain the spreading. These conditions are much more restrictive than those for the model with linear diffusion ($\m=1$) or slow diffusion ($\m>1$). For the latter models, ground states can exist if $h$ is singular at origin (with at most a logarithmic or a power-law singularity for $\m=1$ and $\m>1$, respectively), or if $h$ grows less than exponentially at infinity (superlinear growth for $\m=1$ and no growth rate restriction for $\m>1$). We refer the reader to \cite{FePa2024a, FePa2024b} for the precise statements.
\end{remark}

%%%%%%%%%%

\section{A generalized Carlson--Levin inequality}
\label{sect:C-Lineq}
\setcounter{equation}{0}

This section contains some preparatory results needed to show the existence of global minimizers. Here we assume that the sectional curvatures of the Cartan-Hadamard manifold $M$ satisfy
\begin{equation}
\label{eqn:K-lowerb}
-c_m \leq \mathcal{K}(x;\sigma)\leq 0, 
\end{equation}
for all $x \in M$ and all two-dimensional subspaces $\sigma \subset T_x M$, where $c_m$ is a positive constant.

\subsection{A convexity argument}  Recall Theorem \ref{RCT} and apply it for the manifold $M$ (note that $M$ has non-positive curvature), for $\tilde{M}=\bbr^\dm\simeq T_\p M$, and $i$ the identity map. For two points $x$ and $y$ in $M$, take $c$ to be the geodesic curve joining them. The curve $\tilde{c}$ constructed in Theorem \ref{RCT} joins $\log_\p x$ and $\log_\p y$, and hence, its length is greater than or equal to the length of the straight segment joining $\log_\p x$ and $\log_\p y$. We then infer from Theorem \ref{RCT} that
\begin{equation}
\label{eqn:Rauch-est1}
|\log_\p x - \log_\p y| \leq l(\tilde{c}) \leq l(c) = \dist(x,y).
\end{equation}

By the Law of cosines, we also get
\[
|\log_\p x-\log_\p y|=\sqrt{r_x^2+r_y^2-2r_xr_y\cos\angle(x\p y)},
\]
which combined with \eqref{eqn:Rauch-est1} yields
\begin{align}\label{Rauchest}
\dist(x, y)^2\geq r_x^2+r_y^2-2 r_xr_y\cos\angle(x\p y).
\end{align}
Then, from \eqref{Rauchest} we find
\begin{equation}
\label{distestE}
%\begin{aligned}
\dist(x, y)^2 \geq |r_x-r_y\cos\angle(x\p y)|^2, \qquad \text{ for all } x,y \in M.
%& \geq r_x^2+r_y^2-2 r_xr_y\cos\angle(x\p y) \\
%\end{aligned}
\end{equation}

We will use \eqref{distestE} to prove the following proposition.
\begin{proposition}\label{convineq}
Let $M$ be a Cartan--Hadamard manifold, $\p \in M$ an arbitrary fixed pole, and $\mathcal{H}:[0, \infty)\to[0, \infty)$ be a non-decreasing convex function. Then, for any $\mu\in \calPo(M)$, we have
\begin{equation}
\label{eqn:convineq}
\iint_{M\times M}\mathcal{H}(\dist(x, y))\d\mu(x)\d\mu(y)
\geq \int_M \mathcal{H}(r_x)\d\mu(x).
\end{equation}
\end{proposition}
\begin{proof}
As $\mu\in \calPo(M)$ has Riemannian centre of mass at  $\p$, it holds that
\[
\int_M \log_\p y \, \d \mu(y)=0.
\]
Since $\log_\p\p$ is the zero vector in $T_\p M$, we have
\[
\int_{M\backslash\{\p\}}\log_\p y\, \d\mu(y)=\int_M\log_\p y\, \d\mu(y)=0.
\]
Then, for any $x \neq \p$ we have
\[
0=\int_{M\backslash\{\p\}}\log_\p x\cdot \log_\p y\, \d \mu(y)=r_x\int_{M\backslash\{\p\}}r_y\cos\angle(x\p y)\, \d \mu(y),
\]
and hence,
\begin{align}\label{cossum}
\int_{M\backslash\{\p\}}r_y\cos\angle(x\p y)\, \d \mu(y)=0.
\end{align}

Using Jensen's inequality, the monotonicity of $\mathcal{H}$, and \eqref{distestE}, we get
\begin{align}
\label{eqn:Hineq1}
\begin{aligned}
\int_M\mathcal{H}(\dist(x, y))\d\mu(y)&\geq\mathcal{H}\left(\int_M\dist(x, y)\d\mu(y)\right)\\
&=\mathcal{H}\left(
\int_{M\backslash\{\p\}}\dist(x, y)\d\mu(y)+\int_{\{\p\}}\dist(x, y)\d\mu(y)
\right)
\\
&\geq \mathcal{H}\left(
\int_{M\backslash\{\p\}}|r_x-r_y\cos\angle(x\p y)|\d\mu(y)+\int_{\{\p\}}r_x\, \d\mu(y)
\right).
\end{aligned}
\end{align}
Furthermore, using\eqref{cossum} we find
\begin{equation}
\label{eqn:Harg}
\begin{aligned}
&\int_{M\backslash\{\p\}}|r_x-r_y\cos\angle(x\p y)|\d\mu(y)+\int_{\{\p\}}r_x\d\mu(y)\\[2pt]
&\qquad \geq \left|
\int_{M\backslash\{\p\}}(r_x-r_y\cos\angle(x\p y))\d\mu(y)
\right|+\int_{\{\p\}}r_x\d\mu(y)\\
&\qquad =\int_{M\backslash\{\p\}}r_x\d\mu(y)+\int_{\{\p\}}r_x \d\mu(y)\\
%&\qquad =\int_Mr_x \d\mu(y)\\
&\qquad =r_x.
\end{aligned}
\end{equation}

Substitute \eqref{eqn:Harg} into \eqref{eqn:Hineq1} to conclude that
\begin{align}\label{eqn:Hineq1-1}
\int_M\mathcal{H}(\dist(x, y))\d\mu(y)\geq \mathcal{H}(r_x),
\end{align}
for any $x\neq \p$. For $x=\p$, \eqref{eqn:Hineq1-1} also holds, as
\[
\int_M\mathcal{H}(r_y)\d\mu(y)\geq\int_M\mathcal{H}(0)\d\mu(y)=\mathcal{H}(0).
\]
%\[
%\int_M\mathcal{H}(\dist(x, y))\d\mu(y)=\int_M\mathcal{H}(r_y)\d\mu(y)\geq\int_M\mathcal{H}(0)\d\mu(y)=\mathcal{H}(0)=\mathcal{H}(r_x).
%\]
Hence, \eqref{eqn:Hineq1-1} holds for all $x\in M$. Then, by integrating \eqref{eqn:Hineq1-1} with respect to $\mu(x)$ on $M$ we find \eqref{eqn:convineq}.
%\[
%\iint_{M\times M}\mathcal{H}(\dist(x, y))\d\mu(x)\d\mu(y)\geq \int_M \mathcal{H}(r_x)\d\mu(x),
%\]
\end{proof}

\begin{remark}\label{R4.1} 
If $\mu$ is not a probability measure, \eqref{eqn:convineq} gains an additional term as follows. Assume $\mu\in\mathcal{M}_+(M)\cap L^1(M)$, where $\mathcal{M}_+(M)$ denotes the set of non-negative measures on $M$, satisfies $\displaystyle \|\mu\|_1:=\int_M\d\mu(x)\neq1$, $\displaystyle \int_M r_x\, \d\mu(x)<\infty$, and $\displaystyle \int_M\log_\p x \, \d\mu(x)=0$. Then, we can substitute $\displaystyle \frac{\mu}{\|\mu\|_1}\in\mathcal{P}_{1, \p}(M)$ into Proposition \ref{convineq} to get
\[
\iint_{M\times M}\mathcal{H}(\dist(x, y))\d\mu(x)\d\mu(y)
\geq \left(\int_M \mathcal{H}(r_x)\d\mu(x)\right)\left(\int_M\d\mu(x)\right),
\]
for any non-decreasing convex function $\mathcal{H}:[0, \infty)\to[0, \infty)$.
\end{remark}

%%%%%

\subsection{Generalized Carlson-Levin inequality} The following lemma generalizes the Carlson-Levin inequality \cite[Lemma 5]{carrillo2019reverse} used to study free energy minimization in the Euclidean space $\bbr^\dm$.

\begin{lemma}[Generalized Carlson--Levin inequality]\label{mainineq}
Let $0<q<1$ and $M$ be a $\dm$-dimensional Cartan--Hadamard manifold whose sectional curvatures satisfy \eqref{eqn:K-lowerb}, for some $c_m>0$. Also, let $o \in M$ be a fixed pole. If $\lambda>\frac{(\dm-1)(1-q)}{q}$, then there exists a positive constant $C_1>0$ that depends on $\lambda, q, c_m$, and $\dm$, such that
\[
\int_M\rho(x)^q \dx\leq C_1(\lambda, q, c_m, \dm)\left(\int_M\rho(x) \dx\right)^{(1-p)q}\left(\int_M\left(\frac{\sinh(\sqrt{c_m}r_x)}{\sqrt{c_m}}\right)^\lambda \rho(x) \dx\right)^{pq},
\]
for any $\rho(\cdot) \geq 0$, where $p=\frac{(\dm-1)(1-q)}{\lambda q}\in(0, 1)$. 
\end{lemma}

\begin{proof}
Fix $R>0$ and $\rho(\cdot) \geq 0$. By H\"{o}lder inequality with exponents $\frac{1}{q}$ and $\frac{1}{1-q}$, we get
\begin{equation}
\label{eqn:est-tlr}
\begin{aligned}
\int_{r_x<R}\rho(x)^q \dx&\leq \left(\int_{r_x<R}\rho(x) \dx\right)^q\left(\int_{r_x< R}1\, \dx \right)^{1-q}\\[5pt]
&\leq\left(\int_M\rho(x) \dx \right)^q\left(\dm w(\dm)\int_0^R\left(\frac{\sinh(\sqrt{c_m}\theta)}{\sqrt{c_m}}\right)^{\dm-1}\d\theta\right)^{1-q},
\end{aligned}
\end{equation}
where for the second inequality we used the volume upper bound in Theorem \ref{lemma:Chavel-thms}. Similarly, also by H\"{o}lder inequality with exponents $\frac{1}{q}$ and $\frac{1}{1-q}$, we find
\begin{equation}
\label{eqn:est-tgr}
\begin{aligned}
&\int_{r_x\geq R}\rho(x)^q \dx \\
&\quad \leq \left(\int_{r_x\geq R}
\left(\frac{\sinh(\sqrt{c_m}r_x)}{\sqrt{c_m}}\right)^\lambda\rho(x) \dx
\right)^q
\left(
\int_{r_x\geq R}\left(\frac{\sinh(\sqrt{c_m}r_x)}{\sqrt{c_m}}\right)^{-\frac{\lambda q}{1-q}} \dx
\right)^{1-q}\\[5pt]
& \quad \leq  \left(\int_M
\left(\frac{\sinh(\sqrt{c_m}r_x)}{\sqrt{c_m}}\right)^\lambda\rho(x) \dx
\right)^q
\left(
\dm w(\dm)\int_R^\infty \left(\frac{\sinh(\sqrt{c_m}\theta)}{\sqrt{c_m}}\right)^{\dm-1-\frac{\lambda q}{1-q}}\d\theta\right)^{1-q},
\end{aligned}
\end{equation}
where for the second inequality we used geodesic spherical coordinates $x = \exp_\p(\theta u)$, with $\theta>0$ and $u \in \bbs^{\dm-1}$, to change the variable in the second integral, and reasoned as follows. First note that $(\theta,u)$ are spherical coordinates in the tangent space $T_\p M$, with volume element $\theta^{\dm -1} \d \theta \d \sigma(u)$, where $\d \sigma(u)$ denotes the Riemanian measure on $\bbs^{\dm -1}$. Then, using the Jacobian upper bound in Theorem \ref{lemma:Chavel-thms} we have
\begin{equation}
\label{eqn:dV-bound}
\begin{aligned}
\d V(x) &= J(\exp_\p)(\theta u) \, \theta^{\dm -1} \d \theta \d \sigma(u) \\[2pt]
& \leq  \left(\frac{\sinh(\sqrt{c_m}\theta)}{\sqrt{c_m}}\right)^{\dm-1} \d \theta \d \sigma(u),
\end{aligned}
\end{equation}
from which the second inequality in \eqref{eqn:est-tgr} follows.

%\[
%|\partial B_\theta(\p)|\leq \dm w(\dm)\left(\frac{\sinh \sqrt{c_m}\theta}{\sqrt{c_m}}\right)^{\dm-1}.
%\]

Denote
\begin{align*}
&A(R):=\left(\dm w(\dm)\int_0^R\left(\frac{\sinh(\sqrt{c_m}\theta)}{\sqrt{c_m}}\right)^{\dm-1}\d\theta\right)^{1-q},\\[2pt]
&B(R):=\left(
\dm w(\dm)\int_R^\infty \left(\frac{\sinh(\sqrt{c_m}\theta)}{\sqrt{c_m}}\right)^{\dm-1-\frac{\lambda q}{1-q}}\d\theta\right)^{1-q}.
\end{align*}
Then, by combining \eqref{eqn:est-tlr} and \eqref{eqn:est-tgr} we get
\begin{equation}
\label{eqn:est-rhoq}
\int_M\rho(x)^q \dx\leq A(R)\left(\int_M
\rho(x) \dx
\right)^q+B(R)\left(\int_M
\left(\frac{\sinh(\sqrt{c_m}r_x)}{\sqrt{c_m}}\right)^\lambda\rho(x)\dx
\right)^q.
\end{equation}

Now use
\[
\frac{\sinh(\sqrt{c_m}\theta)}{\sqrt{c_m}}=\frac{e^{\sqrt{c_m}\theta}}{2\sqrt{c_m}}(1-e^{-2\sqrt{c_m}\theta})\leq \frac{e^{\sqrt{c_m}\theta}}{2\sqrt{c_m}}(1-e^{-2\sqrt{c_m}R}), \qquad \forall\theta\in[0, R],
\]
to estimate $A(R)$ as
\begin{equation}
\label{eqn:estA}
\begin{aligned}
A(R)&\leq \left(\dm w(\dm)\int_0^R\left(\frac{e^{\sqrt{c_m}\theta}}{2\sqrt{c_m}}(1-e^{-2\sqrt{c_m}R})\right)^{\dm-1}\d\theta\right)^{1-q}\\[2pt]
%&=\left(\frac{\dm w(\dm)(1-e^{-2\sqrt{c_m}R})^{\dm-1}}{(2\sqrt{c_m})^{\dm-1}}\int_0^Re^{\sqrt{c_m}(\dm-1)\theta}\d\theta\right)^{1-q}\\
&=\left(\frac{\dm w(\dm)(1-e^{-2\sqrt{c_m}R})^{\dm-1}}{(2\sqrt{c_m})^{\dm-1}\sqrt{c_m}(\dm-1)}\left(e^{\sqrt{c_m}(\dm-1)R}-1\right)
\right)^{1-q}\\[2pt]
&\leq\left(\frac{\dm w(\dm)(1-e^{-2\sqrt{c_m}R})^{\dm-1}}{(2\sqrt{c_m})^{\dm-1}\sqrt{c_m}(\dm-1)}e^{\sqrt{c_m}(\dm-1)R}
\right)^{1-q}\\[2pt]
%&=\left(\frac{\dm w(\dm)}{\sqrt{c_m}^{\dm}(\dm-1)} \sinh^{\dm-1}(\sqrt{c_m}R) \right)^{1-q}\\[2pt]
%&=\left(\frac{\dm w(\dm)}{c_m^{\dm/2}(\dm-1)}\right)^{1-q}\sinh^{(1-q)(\dm-1)}(\sqrt{c_m}R)\\
&=\left(\frac{\dm w(\dm)}{\sqrt{c_m}(\dm-1)}\right)^{1-q}\left(\frac{\sinh(\sqrt{c_m}R)}{\sqrt{c_m}}\right)^{(1-q)(\dm-1)}.
\end{aligned}
\end{equation}
Similarly, use
\[
\frac{\sinh(\sqrt{c_m}\theta)}{\sqrt{c_m}}=\frac{e^{\sqrt{c_m}\theta}}{2\sqrt{c_m}}(1-e^{-2\sqrt{c_m}\theta})\geq \frac{e^{\sqrt{c_m}\theta}}{2\sqrt{c_m}}(1-e^{-2\sqrt{c_m}R}), \qquad \forall \theta\in[R, \infty),
\]
to estimate $B(R)$ (recall that $d-1-\frac{\lambda q}{1-q}<0$) as
\begin{equation}
\label{eqn:estB}
\begin{aligned}
B(R)&\leq\left(
\dm w(\dm)\int_R^\infty \left( \frac{e^{\sqrt{c_m}\theta}}{2\sqrt{c_m}}(1-e^{-2\sqrt{c_m}R})\right)^{\dm-1-\frac{\lambda q}{1-q}}\d\theta\right)^{1-q}\\[2pt]
&=\left(
\frac{\dm w(\dm)(1-e^{-2\sqrt{c_m}R})^{\dm-1-\frac{\lambda q}{1-q}}}{(2\sqrt{c_m})^{\dm-1-\frac{\lambda q}{1-q}}}
\int_R^\infty  e^{\sqrt{c_m} \left(\dm-1-\frac{\lambda q}{1-q} \right)\theta}\d\theta
\right)^{1-q}\\[2pt]
&=\left(
\frac{\dm w(\dm)(1-e^{-2\sqrt{c_m}R})^{\dm-1-\frac{\lambda q}{1-q}}}{(2\sqrt{c_m})^{\dm-1-\frac{\lambda q}{1-q}} \sqrt{c_m} \left(
-\dm+1+\frac{\lambda q}{1-q} \right)}
 e^{\sqrt{c_m}\left(\dm-1-\frac{\lambda q}{1-q}\right)R}
\right)^{1-q}\\
%&=\left(
%\frac{\dm w(\dm)}{(\sqrt{c_m})^{\dm-1-\frac{\lambda q}{1-q}}\left(\sqrt{c_m}(
%-\dm+1+\frac{\lambda q}{1-q})\right)}
% \sinh^{\dm-1-\frac{\lambda q}{1-q}}\left(\sqrt{c_m}R\right)
%\right)^{1-q}\\
&=\left(\frac{\dm w(\dm)(1-q)}{\sqrt{c_m}(\lambda q-(1-q)(\dm-1))}\right)^{1-q}\left(\frac{\sinh(\sqrt{c_m}R)}{\sqrt{c_m}}\right)^{(1-q)(d-1)-\lambda q}.
\end{aligned}
\end{equation}

Make the following notations (by the assumption on $\lambda$, we have $\lambda q - (1-q) (\dm-1) >0$):
\begin{equation}
\label{eqn:alphabeta}
\begin{aligned}
%\begin{cases}
\displaystyle\alpha_1&=\left(\frac{\dm w(\dm)}{\sqrt{c_m}(\dm-1)}\right)^{1-q},\\[5pt]
\displaystyle\alpha_2&=\left(\frac{\dm w(\dm)(1-q)}{\sqrt{c_m}(\lambda q-(1-q)(\dm-1))}\right)^{1-q},\\[5pt]
\displaystyle\beta_1&=(1-q)(\dm-1),\\[5pt]
\displaystyle\beta_2&=\lambda q-(1-q)(\dm-1).
%\end{cases}
\end{aligned}
\end{equation}
Note that $\alpha_1, \alpha_2, \beta_1, \beta_2$ are positive constants that depend on $\lambda, q, c_m$, and $\dm$. With these notations, we write \eqref{eqn:estA} and \eqref{eqn:estB} compactly as
\[
A(R)\leq \alpha_1\left(\frac{\sinh(\sqrt{c_m}R)}{\sqrt{c_m}}\right)^{\beta_1}, \qquad\text{ and }\qquad B(R)\leq \alpha_2\left(\frac{\sinh(\sqrt{c_m}R)}{\sqrt{c_m}}\right)^{-\beta_2}, \qquad \forall R>0.
\]

Using the estimates above in \eqref{eqn:est-rhoq}, we have
\begin{equation}
\label{nonoptR}
\begin{aligned}
\int_M \rho(x)^q \dx& \leq \alpha_1\left(\frac{\sinh(\sqrt{c_m}R)}{\sqrt{c_m}}\right)^{\beta_1}\left(\int_M
\rho(x) \dx
\right)^q \\[5pt]
& \quad +\alpha_2\left(\frac{\sinh(\sqrt{c_m}R)}{\sqrt{c_m}}\right)^{-\beta_2}\left(\int_M
\left(\frac{\sinh(\sqrt{c_m}r_x)}{\sqrt{c_m}}\right)^\lambda\rho(x) \dx
\right)^q.
\end{aligned}
\end{equation}
We are interested in finding the optimal $R$ for the right-hand-side of \eqref{nonoptR}. For this purpose, we calculate its derivative with respect to $R$, and set it to zero, to get
\begin{align*}
0&=\alpha_1\beta_1\left(\frac{\sinh(\sqrt{c_m}R)}{\sqrt{c_m}}\right)^{\beta_1-1} \cosh(\sqrt{c_m}R) \left(\int_M
\rho(x) \dx \right)^q\\[5pt]
&\quad -\alpha_2\beta_2\left(\frac{\sinh(\sqrt{c_m}R)}{\sqrt{c_m}}\right)^{-\beta_2-1} \cosh(\sqrt{c_m}R) \left(\int_M
\left(\frac{\sinh(\sqrt{c_m}r_x)}{\sqrt{c_m}}\right)^\lambda\rho(x) \dx \right)^q.
\end{align*}

After simplifying, we find
\begin{equation}
\label{eqn:Reqn}
\left(\frac{\sinh(\sqrt{c_m}R)}{\sqrt{c_m}}\right)^{\beta_1+\beta_2}=
\frac{\alpha_2\beta_2}{\alpha_1\beta_1}\left(
\frac{\int_M
\left(\frac{\sinh(\sqrt{c_m}r_x)}{\sqrt{c_m}}\right)^\lambda\rho(x) \dx}{\int_M
\rho(x) \dx}
\right)^q.
\end{equation}

Let $R_*$ be the unique solution to \eqref{eqn:Reqn} Now substitute $R_*$ into \eqref{nonoptR} to get
\begin{align*}
&\int_M \rho(x)^q\dx\\
&\leq \alpha_1\left(
\frac{\alpha_2\beta_2}{\alpha_1\beta_1}\left(
\frac{\int_M
\left(\frac{\sinh(\sqrt{c_m}r_x)}{\sqrt{c_m}}\right)^\lambda\rho(x)\dx}{\int_M
\rho(x)\dx}
\right)^q
\right)^{\frac{\beta_1}{\beta_1+\beta_2}}
\left(\int_M\rho(x)\dx\right)^q\\[5pt]
&\quad +\alpha_2
\left(
\frac{\alpha_1\beta_1}{\alpha_2\beta_2}\left(
\frac{\int_M
\rho(x)\dx}{\int_M
\left(\frac{\sinh(\sqrt{c_m}r_x)}{\sqrt{c_m}}\right)^\lambda\rho(x)\dx}
\right)^q
\right)^{\frac{\beta_2}{\beta_1+\beta_2}}
\left(\int_M
\left(\frac{\sinh(\sqrt{c_m}r_x)}{\sqrt{c_m}}\right)^\lambda\rho(x)\dx
\right)^q\\[5pt]
&=\left(\alpha_1^{\beta_2}\alpha_2^{\beta_1}\right)^{\frac{1}{\beta_1+\beta_2}}
\left(
\left(\frac{\beta_2}{\beta_1}\right)^{\frac{\beta_1}{\beta_1+\beta_2}} +
\left(\frac{\beta_1}{\beta_2}\right)^{\frac{\beta_2}{\beta_1+\beta_2}}
\right) \\
& \qquad \times 
\left(\int_M\rho(x) \dx\right)^{\frac{\beta_2 q}{\beta_1+\beta_2}}
\left(\int_M \left(\frac{\sinh(\sqrt{c_m}r_x)}{\sqrt{c_m}}\right)^\lambda\rho(x)\dx
\right)^{\frac{\beta_1q}{\beta_1+\beta_2}}\\[5pt]
&=C_1(\lambda, q, c_m, \dm)\left(\int_M\rho(x)\dx\right)^{\frac{\beta_2 q}{\beta_1+\beta_2}}
\left(\int_M \left(\frac{\sinh(\sqrt{c_m}r_x)}{\sqrt{c_m}}\right)^\lambda\rho(x)\dx
\right)^{\frac{\beta_1q}{\beta_1+\beta_2}},
\end{align*}
where
\[
C_1(\lambda, q, c_m, \dm) = \left(\alpha_1^{\beta_2}\alpha_2^{\beta_1}\right)^{\frac{1}{\beta_1+\beta_2}}
\left(
\left(\frac{\beta_2}{\beta_1}\right)^{\frac{\beta_1}{\beta_1+\beta_2}} +
\left(\frac{\beta_1}{\beta_2}\right)^{\frac{\beta_2}{\beta_1+\beta_2}}
\right).
\]

Using the expressions for $\beta_1$ and $\beta_2$ from \eqref{eqn:alphabeta}, we have
\[
\frac{\beta_1}{\beta_1+\beta_2}=\frac{(1-q)(\dm-1)}{\lambda q}, \qquad\text{ and }\qquad\frac{\beta_2}{\beta_1+\beta_2}=1-\frac{(1-q)(\dm-1)}{\lambda q}.
\]
Finally, we find
\[
\int_M\rho(x)^q\dx\leq C_1\left(\int_M\rho(x)\dx\right)^{(1-p)q}\left(\int_M\left(\frac{\sinh(\sqrt{c_m}r_x)}{\sqrt{c_m}}\right)^\lambda\rho(x)\dx\right)^{pq},
\]
where $p=\frac{(1-q)(\dm-1)}{\lambda q}\in(0,1)$..
\end{proof}

\begin{remark}\label{R4.2}
By combining Lemma \ref{mainineq} and Remark \ref{R4.1} one gets
\[
\int_M\rho(x)^q \dx\leq C_1\left(\int_M\rho(x) \dx\right)^{(1-2p)q}\left(\iint_{M\times M}\left(\frac{\sinh(\sqrt{c_m}\dist(x, y))}{\sqrt{c_m}}\right)^\lambda \rho(x)\rho(y)\dx \dy\right)^{pq}
\]
for all $\rho$ that satisfy the conditions in Remark \ref{R4.1}. This can be further transformed into
\begin{multline}
C_1^{-1/pq} \left( \int_M \rho(x)^q \dx \right)^{\frac{1}{pq}} \left(\int_M\rho(x) \dx\right)^{-\frac{(1-2p)q}{pq}} \\
\leq \iint_{M\times M}\left(\frac{\sinh(\sqrt{c_m}\dist(x, y))}{\sqrt{c_m}}\right)^\lambda \rho(x)\rho(y)\dx \dy.
\end{multline}

When $c_m\to0+$, we have
\[
\lim_{c_m\to0+}\frac{\sinh(\sqrt{c_m}\dist(x, y))}{\sqrt{c_m}}=\dist(x,y).
\]
Therefore, on $M=\bbr^\dm$, we can infer the inequality
\begin{align*}
C_1^{-1/pq} \left( \int_{\bbr^\dm} \rho(x)^q \d x \right)^{\frac{1}{pq}} \left(\int_{\bbr^\dm}\rho(x) \d x\right)^{-\frac{(1-2p)q}{pq}}
\leq \iint_{\bbr^\dm\times \bbr^\dm}\dist(x,y)^\lambda \rho(x)\rho(y)\d x \d y.
\end{align*}
which is a reversed HLS inequality with a non-optimal constant (see \cite[Theorem 1]{carrillo2019reverse}).
\end{remark}

%%%%%%%%%%

\section{Existence of a global minimizer}
\label{sect:existence}
\setcounter{equation}{0}

In this section, we assume that the sectional curvatures $M$ satisfy \eqref{eqn:K-lowerb}, for some $c_m>0$, and investigate the existence of global energy minimizers in $\calPo(M)$ (see \eqref{eqn:calP-1o}), with $\p$ being an arbitrary fixed pole in $M$. 

Specifically, we assume that $h$ satisfies 
\begin{equation}
\label{eqn:condH0}
h(0)=0,
\end{equation}
and
\begin{equation}
\label{condiH}
\liminf_{\theta\to\infty}\frac{h(\theta)}{\exp(\sqrt{c_m}\lambda \theta)}>0,\qquad\text{ for some }\lambda>\frac{(\dm-1)(1-\m)}{\m}.
\end{equation}
Both conditions on $h$ in \eqref{eqn:condH0} and \eqref{condiH} complement those from Theorem \ref{thm:nonexist}, when nonexistence of minimizers occurs. In particular, a necessary condition for existence is that $h$ does not blow-up at origin; since the interaction potential is defined up to a constant, we set $h$ to be $0$ at origin. Also, the interaction potential needs to grow at least exponentially, or else an energy minimizer does not exist. This is reflected in the condition on $h$ in \eqref{condiH}.

\begin{proposition}[Tightness of the minimizing sequences]
\label{prop:tightness}
Let $M$ be a $\dm$-dimensional Cartan-Hadamard manifold whose sectional curvatures satisfy \eqref{eqn:K-lowerb}, for some $c_m>0$, and let $\p\in M$ be a fixed pole. Also, assume $0<\m<1$ and that $h$ is lower semi-continuous, non-decreasing, and satisfies \eqref{eqn:condH0} and \eqref{condiH}. Then, any minimizing sequence $\{\mu_n\}_{n \geq 1}\subset \mathcal{P}_{1, \p}(M)$ of the functional $E[\cdot]$ is tight.
\end{proposition}

\begin{proof}
Consider a minimizing sequence $\{\mu_n\}_{n \geq 1}\subset \mathcal{P}_{1, \p}(M)$ of $E[\cdot]$, and for each $n \geq 1$, denote by $\rho_n$ the absolutely continuous part of $\mu_n$. Then,
\begin{equation}
\label{ineq:Gmun}
\begin{aligned}
%E[\mu_1]&\geq 
E[\mu_n]&=-\frac{1}{1-\m}\int_M \rho_n(x)^\m\d x+\frac{1}{2}\iint_{M\times M} h(\dist(x, y))\d\mu_n(x)\d \mu_n(y)\\
&\geq -\frac{C_1}{1-\m}\left(\int_M \left(\frac{\sinh(\sqrt{c_m}r_x)}{\sqrt{c_m}}\right)^\lambda \rho_n(x)\d x\right)^{p\m} \\
& \quad +\frac{1}{4}\iint_{M\times M} h(\dist(x, y))\d\mu_n(x)\d \mu_n(y) +\frac{1}{4}\iint_{M\times M} h(\dist(x, y))\d\mu_n(x)\d \mu_n(y),
\end{aligned}
\end{equation}
where for the inequality we used Lemma \ref{mainineq} together with the fact that $\int_M \rho_n(x) \d x \leq \int_M \d \mu_n(x) =1$, with $\lambda$ being the constant in \eqref{condiH} and
\[
p=\frac{(\dm-1)(1-\m)}{\lambda \m} \in (0,1).
\]
Note that $C_1>0$ depends on $\lambda, \m, c_m$, and $\dm$, but not on $\mu_n$.

By \eqref{condiH} and 
\[
\lim_{\theta\to\infty}\frac{\exp(\sqrt{c_m}\lambda\theta)}{\left(\frac{\sinh(\sqrt{c_m}\theta)}{\sqrt{c_m}}\right)^{\lambda}}
=\left(\lim_{\theta\to\infty}
\frac{\sqrt{c_m}\exp(\sqrt{c_m}\theta)}{\sinh(\sqrt{c_m}\theta)}
\right)^\lambda
=\lim_{\theta\to\infty}\left(
\frac{2\sqrt{c_m}}{1-\exp(-2\sqrt{c_m}\theta)}
\right)^\lambda=2^\lambda c_m^{\lambda/2},
\]
we also have
\begin{equation}
\label{eqn:hosinh}
\liminf_{\theta\to\infty}\frac{h(\theta)}{\left(\frac{\sinh(\sqrt{c_m}\theta)}{\sqrt{c_m}}\right)^\lambda}>0.
\end{equation}
Therefore, there exist some constants $\gamma_1>0$ and $\gamma_2$ (which only depend on $h$, $\lambda$ and $c_m$) such that
\begin{align}\label{h-sinhest}
h(\theta)\geq \gamma_1\left(\frac{\sinh(\sqrt{c_m}\theta)}{\sqrt{c_m}}\right)^{\lambda}+\gamma_2, \qquad\forall \theta>0.
\end{align}
Indeed, we can set $\gamma_1$ to be a positive constant less than $2^\lambda c_m^{\lambda/2}$ and 
\[
\gamma_2:=\inf_{\theta\geq0}\left(h(\theta)-\gamma_1\left(\frac{\sinh(\sqrt{c_m}\theta)}{\sqrt{c_m}}\right)^\lambda\right).
\]

We substitute \eqref{h-sinhest} into \eqref{ineq:Gmun} to get
\begin{align*}
E[\mu_n]&\geq -\frac{C_1}{1-\m}\left(\int_M \left(\frac{\sinh(\sqrt{c_m}r_x)}{\sqrt{c_m}}\right)^\lambda \rho_n(x)\d x\right)^{p\m} \\
& \quad +\frac{\gamma_1}{4}\iint_{M\times M} \left(\frac{\sinh(\sqrt{c_m}\dist(x, y))}{\sqrt{c_m}}\right)^\lambda\d\mu_n(x)\d \mu_n(y) +\frac{\gamma_2}{4} +\frac{1}{4}\iint_{M\times M} h(\dist(x, y))\d\mu_n(x)\d \mu_n(y).
\end{align*}
Then, using Proposition \ref{convineq} with the convex function 
\begin{equation}
\label{eqn:H-sinh}
\mathcal{H}(\theta)=\left(\frac{\sinh\sqrt{c_m}\theta}{\sqrt{c_m}}\right)^\lambda,
\end{equation}
we further find
\begin{equation}
\label{ineqGmu1}
\begin{aligned}
E[\mu_n]&\geq -\frac{C_1}{1-\m}\left(\int_M \left(\frac{\sinh(\sqrt{c_m}r_x)}{\sqrt{c_m}}\right)^\lambda \rho_n(x)\d x\right)^{p \m} \\
& \quad +\frac{\gamma_1}{4}\int_{M} \left(\frac{\sinh(\sqrt{c_m}r_x)}{\sqrt{c_m}}\right)^\lambda\d\mu_n(x) +\frac{\gamma_2}{4}
+\frac{1}{4}\iint_{M\times M} h(\dist(x, y))\d\mu_n(x)\d \mu_n(y)\\
&\geq -\frac{C_1}{1-\m}\left(\int_M \left(\frac{\sinh(\sqrt{c_m}r_x)}{\sqrt{c_m}}\right)^\lambda \rho_n(x)\d x\right)^{p \m} \\
& \quad +\frac{\gamma_1}{4}\int_{M} \left(\frac{\sinh(\sqrt{c_m}r_x)}{\sqrt{c_m}}\right)^\lambda\rho_n(x)\d x +\frac{\gamma_2}{4} +\frac{1}{4}\iint_{M\times M} h(\dist(x, y))\d\mu_n(x)\d \mu_n(y),
\end{aligned}
\end{equation}
where for the second inequality we used that $\rho_n$ is the absolutely continuous part of $\mu_n$.

Define $f:[0, \infty)\to \bbr$ by
\[
f(X)=-\frac{C_1}{1-\m}X^{p \m}+\frac{\gamma_1}{4}X. 
\]
As $0<p \m<1$ and $C_1,\gamma_1>0$, $f$ has a global minimum on $[0,\infty)$, and hence,
\[
f(X)\geq C_2, \qquad\forall X\in[0, \infty),
\]
for some $C_2 \in \bbr$. Using this property in \eqref{ineqGmu1} for 
\[
X=\int_{M} \left(\frac{\sinh(\sqrt{c_m}r_x)}{\sqrt{c_m}}\right)^{\lambda}\rho_n(x)\d x,
\]
we find
\begin{equation}
\label{eqn:E-lb-h}
E[\mu_n]\geq C_2+\frac{\gamma_2}{4}+\frac{1}{4}\iint_{M\times M}h(\dist(x, y))\d\mu_n(x)\d \mu_n(y).
\end{equation}
We point out that the constants in \eqref{eqn:E-lb-h} depend on $h,\lambda, \m, c_m$, and $\dm$, but not on $\mu_n$.

Since $\mu_n$ is a minimizing sequence (and hence, $E[\mu_1] \geq E[\mu_n]$),  we infer from \eqref{eqn:E-lb-h} that
\[
\iint_{M\times M}h(\dist(x, y))\d\mu_n(x)\d \mu_n(y)
\]
is uniformly bounded above in $n\in\mathbb{N}$. From \eqref{h-sinhest}, we get that
\[
\iint_{M\times M}\left(\frac{\sinh(\sqrt{c_m}\dist(x,y))}{\sqrt{c_m}}\right)^\lambda\d\mu_n(x) \d \mu_n(y)
\]
is also uniformly bounded in $n\in\mathbb{N}$; denote by $C_3$ a uniform upper bound for it.  From the convexity argument in Proposition \ref{convineq} (see \eqref{eqn:convineq} and \eqref{eqn:H-sinh}), we then find
\begin{equation}
\label{eqn:C3ineq}
\begin{aligned}
C_3 &\geq \iint_{M\times M}\left(\frac{\sinh(\sqrt{c_m}\dist(x,y))}{\sqrt{c_m}}\right)^\lambda\d\mu_n(x) \d \mu_n(y) \\
& \geq \int_{M} \left(\frac{\sinh(\sqrt{c_m}r_x)}{\sqrt{c_m}}\right)^{\lambda}\d \mu_n(x),\qquad \forall n\in\mathbb{N}.
\end{aligned}
\end{equation}

Fix $R>0$. From \eqref{eqn:C3ineq}, we further estimate
\begin{align*}
C_3
%&\geq \int_{M} \left(\frac{\sinh(\sqrt{c_m}r_x)}{\sqrt{c_m}}\right)^{\lambda}\d\mu_n(x)\\
&\geq\int_{r_x\geq R} \left(\frac{\sinh(\sqrt{c_m}r_x)}{\sqrt{c_m}}\right)^{\lambda}\d\mu_n(x)\\
&\geq \int_{r_x\geq R} \left(\frac{\sinh(\sqrt{c_m}R)}{\sqrt{c_m}}\right)^{\lambda}\d\mu_n(x)\\
&= \left(\frac{\sinh(\sqrt{c_m}R)}{\sqrt{c_m}}\right)^{\lambda}\int_{r_x\geq R}\d\mu_n(x), \qquad \forall n\in\mathbb{N}.
\end{align*}
This implies
\begin{align}\label{C3inq2}
\int_{r_x\geq R}\d \mu_n(x)\leq \frac{C_3}{\left(\frac{\sinh(\sqrt{c_m}R)}{\sqrt{c_m}}\right)^{\lambda}}, \qquad \forall n\in\mathbb{N},
\end{align}
where the constant $C_3$ does not depend on $R$ or $n$. We infer from here that $\{\mu_n\}_{n \geq 1}$ is tight. 
\end{proof}

\begin{remark}
\label{rmk:E-lb}
As an immediate consequence of the proof of Proposition \ref{prop:tightness}, for any minimizing sequence $\{\mu_n\}_{n \geq 1}\subset \mathcal{P}_{1, \p}(M)$ of the energy, it holds that
\begin{equation}
\label{eqn:E-lb}
E[\mu_n] \geq \widetilde{C}_1 + \widetilde{C}_2 \, \calW_1(\mu_n, \delta_\p), \qquad \text{ for all } n \geq 1,
\end{equation}
with constants $\widetilde{C}_1$ and $\widetilde{C}_2$ that are independent of $\mu_n$. Indeed, by \eqref{eqn:E-lb-h} and \eqref{h-sinhest}, we get 
\begin{align*}
E[\mu_n]& \geq C_2+\frac{\gamma_2}{4}+\frac{1}{4}\iint_{M\times M}  \left( \gamma_1\left(\frac{\sinh(\sqrt{c_m}\theta)}{\sqrt{c_m}}\right)^{\lambda}+\gamma_2 \right) \d\mu_n(x)\d \mu_n(y) \\
&\geq C_2 +  \frac{\gamma_2}{2} + \frac{\gamma_1}{4} \int_{M} \left(\frac{\sinh(\sqrt{c_m}r_x)}{\sqrt{c_m}}\right)^{\lambda}\d \mu_n(x),
\end{align*}
where for the second inequality we used Proposition \ref{convineq} (see also \eqref{eqn:C3ineq}). The conclusion then follows by noting that there exist constants $\widetilde{\gamma_1}$ and $\widetilde{\gamma_2}$ (that depend on $c_m$ and $\lambda$) such that
\[
 \left(\frac{\sinh(\sqrt{c_m} \theta)}{\sqrt{c_m}}\right)^{\lambda} \geq \widetilde{\gamma_1} \theta + \widetilde{\gamma_2}, \qquad \text{ for all } \theta>0,
\]
and that
\[
\int_M r_x \, \d \mu_n(x)  = \calW_1(\mu_n,\delta_\p).
\]
\end{remark}

\begin{proposition}[Lower semi-continuity of the energy]\label{lscEnergy}
Let $M$ and $h$ satisfy the assumptions of Proposition \ref{prop:tightness}, and $\p$ be a fixed pole in $M$. Then, for $0<\m<1$, $E[\cdot]$ is lower semi-continuous on minimizing sequences, i.e., for any minimizing sequence $\{\mu_n\}_{n \geq 1}\subset\mathcal{P}_{1, \p}(M)$ of $E$ that converges weakly to $\mu_0$, we have
\[
\liminf_{n\to\infty}E[\mu_n]\geq E[\mu_0].
\]
\end{proposition}

\begin{proof}
We investigate the lower semi-continuity of the entropy part and the interaction energy part, separately. To do so, let  $\{\mu_n\}_{n \geq 1}\subset \mathcal{P}_{1, \p}(M)$ be a minimizing sequence of $E[\cdot]$ which converges weakly to some measure $\mu_0 \in \calP(M)$.  Denote by $\rho_n$ and $\rho_0$ the absolutely continuous parts of $\mu_n$ and $\mu_0$, respectively. 
\smallskip

\noindent \underline{Part 1.} ({\em Lower semi-continuity of the entropy}) Fix a radius $r>0$ and consider the Riemannian volume measure $\dV$ on $\overline{B_r(\p)}$. By \cite[Proposition 7.7]{Santambrogio2015}, given a convex lower semi-continuous function $U:\bbr_+\to\bbr$, the functional 
\[
\mathcal{F}(\mu)=\int_{\overline{B_r(\p)}} U(\rho(x))\dV (x)+L\mu^s\bigl(\overline{B_r(\p)} \bigr)
\]
is lower semi-continuous with respect to weak convergence. Here, $\mu=\rho\, \dV +\mu^s$ is decomposed into its absolutely continuous and singular parts, and $L:=\lim_{t\to\infty}f(t)/t\in\bbr\cup\{+\infty\}$.

Apply this result for $f(x)=\frac{x^\m}{\m-1}$, (note that since $0<\m<1$, $f$ is a convex lower semi-continuous function that satisfies $L=\lim_{t\to\infty}f(t)/t=0$), to conclude that
\[
\frac{1}{\m-1} \int_{\overline{B_r(\p)}}\rho(x)^\m \dx \qquad \text{ is lower semi-continuous}.
\]
Hence,
\begin{align}\label{entint}
\liminf_{n\to\infty}\left(\frac{1}{\m-1} \int_{\overline{B_r(\p)}}\rho_n(x)^\m \, \dx\right)\geq \frac{1}{\m-1}\int_{\overline{B_r(\p)}}\rho_0(x)^\m \, \dx.
\end{align}

Now, apply Lemma \ref{mainineq} for $\rho_n \cdot \ind_{_{B_r(\p)^c}}$ ($\rho_n$ restricted to the complement of the ball $B_r(\p)$), to get
\[
\int_{B_r(\p)^c}\rho_n(x)^\m \, \dx\leq C_1\left(\int_{B_r(\p)^c}\rho_n(x) \, \dx \right)^{(1-p)\m}\left(\int_{B_r(\p)^c}\left(\frac{\sinh(\sqrt{c_m}r_x)}{\sqrt{c_m}}\right)^{\lambda}\rho_n(x) \, \dx \right)^{p \m}.
\]
The constant $C_1>0$ depends on $\lambda, \m, c_m$, and $\dm$, but not on $n$ or $r$. Furthermore, from the above we have
\begin{align}
\begin{aligned}\label{entext}
\int_{B_r(\p)^c}\rho_n(x)^\m \dx&\leq 
%C_1\left(\int_{B_r(\p)^c}\d\mu_n(x)\right)^{(1-p)\m}\left(\int_{B_r(\p)^c}\left(\frac{\sinh(\sqrt{c_m}r_x)}{\sqrt{c_m}}\right)^{\lambda}\d\mu_n(x)\right)^{p\m}\\
C_1\left(\int_{B_r(\p)^c}\d\mu_n(x)\right)^{(1-p) \m}\left(\int_M\left(\frac{\sinh(\sqrt{c_m}r_x)}{\sqrt{c_m}}\right)^{\lambda}\d\mu_n(x)\right)^{p \m}\\
%&\leq C_1\left(\int_{B_r(\p)^c}\d\mu_n(x)\right)^{(1-p) \m}C_3^{p \m}\\
&\leq C_1\left(\frac{C_3}{\left(\frac{\sinh(\sqrt{c_m}r)}{\sqrt{c_m}}\right)^\lambda}\right)^{(1-p) \m}C_3^{p \m} \\ & =C_1C_3^\m\left(\frac{\sqrt{c_m}}{\sinh(\sqrt{c_m}r)}\right)^{\lambda(1-p) \m},
\end{aligned}
\end{align}
where $C_3$ is the constant (independent of $n$) introduced in the proof of Proposition \ref{prop:tightness} -- see \eqref{eqn:C3ineq}. In \eqref{entext}, for the first inequality we used that $\rho_n$ is the absolutely continuous part of $\mu_n$ and that $B_r(\p)^c\subset M$, and for the second inequality we used \eqref{C3inq2} and \eqref{eqn:C3ineq} for each of the integrals, respectively.

From \eqref{entint} and \eqref{entext}, we then have
\begin{align*}
\frac{1}{\m-1}\int_M\rho_0(x)^\m \, \dx & \leq \frac{1}{\m-1}\int_{\overline{B_r(\p)}}\rho_0(x)^\m \, \dx\\
& \leq \liminf_{n\to\infty}\left(\frac{1}{\m-1}\int_{\overline{B_r(\p)}}\rho_n(x)^\m \,\dx\right)\\
& \leq \liminf_{n\to\infty}\left( \frac{1}{\m-1} \int_{\overline{B_r(\p)}}\rho_n(x)^\m \, \dx \right. \\
& \quad \left. + \frac{1}{1-\m} C_1C_3^\m\left(\frac{\sqrt{c_m}}{\sinh(\sqrt{c_m}r)}\right)^{\lambda(1-p)\m}- \frac{1}{1-\m} \int_{B_r(\p)^c}\rho_n(x)^\m \dx  \right)\\
&= \liminf_{n\to\infty}\left(\frac{1}{\m-1} \int_M\rho_n(x)^\m\ dx\right)+ \frac{1}{1-\m} C_1C_3^\m \left(\frac{\sqrt{c_m}}{\sinh(\sqrt{c_m}r)}\right)^{\lambda(1-p)\m}.
\end{align*}
Since $r>0$ was arbitrary, one can send $r\to\infty$ in the above, to find
\begin{equation}
\label{eqn:lsc-entropy}
\liminf_{n\to\infty}\left(\frac{1}{\m-1}\int_M\rho_n(x)^\m\d x\right) \geq \frac{1}{\m-1} \int_M\rho_0(x)^\m\d x.
\end{equation}
\smallskip

\noindent \underline{Part 2.} ({\em Lower semi-continuity of the interaction energy}) The lower semi-continuity of the interaction component can be shown as in \cite{FePa2024b}, by using the Riemannian logarithm $\log_\p$ to map $M$ onto the tangent space $T_\p M$, and then use properties of the interaction energy on Euclidean spaces. Indeed, denote by $f:M\to T_\p M$ the Riemannian logarithm map at $\p$, i.e.,
\begin{equation}
\label{eqn:Rlog}
f(x)=\log_\p x,\qquad \text{ for all }x\in M.
\end{equation}
The inverse $f^{-1}:T_\p M\to M$ is the Riemannian exponential map $\exp_\p$,  which on Cartan-Hadamard manifolds is a global diffeomorphism.

By a change of variable $x=\exp_\p u$, $y=\exp_\p v$, we can write the interaction energy of a probability measure $\mu$ as 
\begin{align*}
 \iint_{M\times M}h(\dist(x, y)) \d \mu(x) \d \mu(y) =\iint_{T_\p M\times T_\p M}h(\dist(\exp_\p u,\exp_\p v )) \d f_\#\mu(u) \d f_\#\mu(v).
\end{align*}
The right-hand-side above can be interpreted as the interaction energy of $f_\# \mu$ on $T_\p M\simeq \mathbb{R}^\dm$ with the interaction potential $\tilde{h} : T_\p M \times T_\p M \to \mathbb{R}^\dm$ given by
\[
\tilde{h}(u, v)=h(\dist(\exp_\p u,\exp_\p v )).
\]
Note that $\tilde{h}$ is a lower semi-continuous function, as $h$ is lower semi-continuous and $\exp_\p$ is a diffeomorphism.

Since $\mu_n \rightharpoonup \mu_0$ weakly as $n \to \infty$, it is straightforward to show by a change of variable that we also have $f_\# \mu_{n} \rightharpoonup f_\# \mu_0$ weakly as $n \to \infty$. The interaction energy on $\bbr^\dm$ is lower semi-continuous with respect to weak convergence, provided the interaction potential is lower semi-continuous and bounded from below \cite[Proposition 7.2]{Santambrogio2015}. Putting together these facts, we then find
\begin{equation}
\label{eqn:lsc-interaction}
\begin{aligned}
\liminf_{n \to \infty }\iint_{M\times M}h(\dist(x, y)) \d \mu_n(x) \d \mu_n(y) &= \liminf_{n \to \infty} \iint_{T_\p M\times T_\p M}\tilde{h}(u,v) \d f_\#\mu_n(u) \d f_\#\mu_n(v) \\[2pt]
& \geq \iint_{T_\p M\times T_\p M}\tilde{h}(u,v) \d f_\#\mu_0(u) \d f_\#\mu_0(v) \\[2pt]
&= \iint_{M\times M}h(\dist(x, y))\d \mu_0(x) \d \mu_0(y).
\end{aligned}
\end{equation}

Finally, the lower semi-continuity of the energy $E[\cdot]$ now follows from \eqref{eqn:lsc-entropy} and \eqref{eqn:lsc-interaction}.
\end{proof}

In the following result we will use the notation $B_R(\delta_\p)$ for the geodesic ball in the space $(\calP_1(M),\calW_1)$, of radius $R$ and centre at $\delta_\p$. In the paper we have used a similar notation for the open ball $B_r(\p)$ of radius $r$ and centre at $\p$, in the geodesic space $(M,\dist)$. Nevertheless, the different spaces in which these geodesic balls lie in, will be clear from the context. 

\begin{lemma}[Conservation of centre of mass \cite{FePa2024b}]
\label{lemma:CM}
Let $M$ be a Cartan-Hadamard manifold, $\p \in M$ a fixed pole, $R>0$ a fixed radius, and $\{\mu_n\}_{n\geq 1}$ a sequence in $\overline{B_R(\delta_\p)}$ such that all $\mu_n$ ($n\geq 1$) have centre at mass at $\p$ and $\mu_n$ converges weakly to $\mu_0$ as $n \to \infty$. Also assume that
\[
\int_M \phi(r_x) \, \d \mu_n(x)
\]
is uniformly bounded from above, where $\phi$ is a function with superlinear growth at infinity, i.e., $\lim_{\theta \to \infty} \frac{\phi(\theta)}{\theta} = \infty$. Then, $\mu_0 \in \calP_{1,\p}(M)$ (in particular, $\mu_0$ has centre of mass at $\p$).
\end{lemma}
\begin{proof}
A variant of this result was stated and proved in \cite[Lemma 5.5]{FePa2024b}. For completeness, we include a proof here as well.

First note that $\mu_n$ has uniformly bounded $1$-moments, as
\[
\int_M \dist(x, \p) \, \d \mu_n(x)  = \calW_1(\mu_n,\delta_\p) \leq R.
\]
Since convergence in $\calW_1$ is equivalent to weak convergence and uniform $1$-integrability \cite{AGS2005}, we infer that $\mu_n$ also converges to $\mu_0$ in $(\calP_1(M),\calW_1)$. Hence, $\mu_0 \in \calP_1(M)$, and more specifically, $\mu_0 \in \overline{B_R(\delta_\p)}$. It remains to show that $\rho_0$ has centre of mass at $\p$.

Fix an arbitrary unit tangent vector $v\in T_\p M$. As $\mu_0\in\overline{B_R(\delta_\p)}$, we have
\[
\left| \int_{M}  \log_\p x\cdot v \, \d \mu_0(x) \right|  \leq \|v\| \int_{M} r_x \, \d \mu_0(x) \leq R.
\]
Fix $\epsilon>0$ arbitrary small. Since the integral $\displaystyle \int_{M}  \log_\p x\cdot v \, \d \mu_0(x) $ is convergent,  there exists $r_1>0$ such that
\begin{equation}
\label{eqn:iint-rho0-log}
\left |\int_{r_x> r}\log_\p x\cdot v \, \d \mu_0(x) \right| < \epsilon, \qquad \text{ for any } r>r_1.
\end{equation}

On the other hand, since $\phi$ grows superlinearly at infinity, there exists $r_2>0$ such that
\[
0< \frac{\theta}{\phi(\theta)} <\epsilon, \qquad \text{ for any } r>r_2. 
\]
Then, for any $r>r_2$ and $n \geq 1$, we have
\begin{equation}
\label{eqn:iint-rhok-log}
\begin{aligned}
\left|\int_{r_x> r} \log_\p x\cdot v \, \d \mu_n(x) \right|&\leq \int_{r_x> r }r_x \, \d \mu_n(x) \\
& = \int_{r_x> r} \phi(r_x)\left(\frac{r_x}{\phi(r_x)}\right) \d  \mu_n(x) \\[2pt]
&\leq \left\|\frac{\theta}{\phi(\theta)}\right\|_{L^\infty((r, \infty))}\times \int_M  \phi(r_x)\d \mu_n(x) \\[2pt]
& \leq U \epsilon,
\end{aligned}
\end{equation}
where $U$ denotes a uniform upper bound of $\int_M \phi (r_x) \d \mu_n(x)$. 

By combining \eqref{eqn:iint-rho0-log} and \eqref{eqn:iint-rhok-log}, for any $r>\max(r_1, r_2)$, we get
\begin{equation}
\label{eqn:iint-CMdiff}
\begin{aligned}
\left|\int_{M} \log_\p x \cdot v \, \d \mu_n(x)   -\int_{M} \log_\p x\cdot v \, \d \mu_0(x) \right| & \leq \left|\int_{r_x\leq r } \log_\p x\cdot v\, \d \mu_n(x) -\int_{r_x\leq r } \log_\p x\cdot v \, \d \mu_0(x) \right|\\[2pt]
& \quad +\left|\int_{r_x> r} \log_\p x\cdot v\, \d \mu_n(x)-\int_{r_x> r }\log_\p x\cdot v \, \d \mu_0(x) \right|\\[2pt]
& \leq \left|\int_{r_x\leq r }\log_\p x\cdot v \, \d \mu_n(x)-\int_{r_x\leq r } \log_\p x\cdot v \, \d \mu_0(x) \right| \\[3pt]
&\quad +(U+1)\epsilon.
\end{aligned}
\end{equation}
Since $\{x:r_x\leq r\}$ is a bounded set, by continuity of the Riemannian logarithm we have
\[
\lim_{n\to \infty} \int_{r_x\leq r } \log_\p x\cdot v \, \d \mu_n(x) = \int_{r_x\leq r } \log_\p x\cdot v \, \d \mu_0(x).
\]

Finally, letting $n \to\infty$ in \eqref{eqn:iint-CMdiff}, we find
\[
\limsup_{n \to\infty}\left|\int_{M} \log_\p x\cdot v \, \d \mu_n(x)-\int_{M}\log_\p x\cdot v \, \d \mu_0(x) \right|\leq (U+1)\epsilon,
\]
for any $\epsilon>0$. From here we infer that
\[
\lim_{n \to\infty}\left|\int_{M} \log_\p x\cdot v \, \d \mu_n(x) -\int_{M} \log_\p x\cdot v \, \d \mu_0(x) \right|=0,
\]
and hence,
\[
\lim_{n \to\infty}\int_{M} \log_\p x\cdot v \, \d \mu_n(x)=\int_{M}\log_\p x\cdot v \, \d \mu_0(x).
\]

Since $\mu_n$ has centre at mass at $\p$, we have $\int_M \log_\p x \cdot v \, \d \mu_n(x)=0$ for all $n \geq 1$, and consequently, 
\[
\int_M \log_\p x \cdot v \, \d \mu_0(x) =0.
\]
Since the constant unit vector $v$ is arbitrary, we infer that $\mu_0$ has centre of mass at $\p$.
\end{proof}

The main result of this section is the following theorem.
\begin{theorem}[Existence of a global minimizer]
\label{thm:existence}
Let $M$ be a $\dm$-dimensional Cartan-Hadamard manifold with sectional curvatures that satisfy \eqref{eqn:K-lowerb}, for some $c_m>0$, and let $\p\in M$ be a fixed pole. Also, let $0<\m<1$ and assume that $h$ is lower semi-continuous, non-decreasing, and satisfies \eqref{eqn:condH0} and \eqref{condiH}. Then, the energy functional $E[\cdot]$ given by \eqref{eqn:energy} has a global minimizer in $\calP_{1,\p}(M)$.
\end{theorem}

\begin{proof}
The proof follows essentially from the previous considerations. The energy is bounded below on $\calP_{1,\p} (M)$ -- see Remark \ref{rmk:E-lb}. Hence, define
\[
E_0:=\inf_{\mu\in \calP_{1,\p}(M)} E[\mu],
\]
and consider a minimizing sequence $\{\mu_n\}_{n\geq 1}\subset \mathcal{P}_{1, \p}(M)$ of $E[\cdot]$, i.e.,
\begin{equation}
\label{eqn:limEmun}
\lim_{n \to \infty} E[\mu_n] = E_0.
\end{equation}
Without loss of generality, we can assume that $E[\mu_n]$ is decreasing.  

By Proposition \ref{prop:tightness}, any minimizing sequence is tight, so by Prokhorov's theorem $\{\mu_n\}_{n\geq 1}$ has a subsequence $\{\mu_{n_k}\}_{k\geq 1}$ which converges weakly to $\mu_0 \in\mathcal{P}(M)$. We claim that the limit $\mu_0$ is also in $\mathcal{P}_{1, \p}(M)$, as inferred from Lemma \ref{lemma:CM}. Indeed, $\calW_1(\mu_{n_k},\delta_\p)$ is uniformly bounded from above by Remark \ref{rmk:E-lb} (see \eqref{eqn:E-lb} and also note that $\displaystyle E[\mu_1] \geq E[\mu_{n_k}]$). Also, from the proof of Proposition \ref{prop:tightness}, we know that 
\[
\int_{M} \left(\frac{\sinh(\sqrt{c_m}r_x)}{\sqrt{c_m}}\right)^{\lambda}\d \mu_n(x)
\]
is uniformly bounded from above -- see  \eqref{eqn:C3ineq}. Then, by using Lemma \ref{lemma:CM} with $\phi(r_x) = \left(\frac{\sinh(\sqrt{c_m}r_x)}{\sqrt{c_m}}\right)^{\lambda}$ (note that $\phi$ has superlinear growth at infinity, as required in the assumptions), we infer that $\mu_0 \in \mathcal{P}_{1, \p}(M)$.

Then, from \eqref{eqn:limEmun} and Proposition \ref{lscEnergy} we get
\[
E_0 =\liminf_{k\to\infty}E[\mu_{n_k}]\geq E[\mu_0]\geq E_0.
\]
Finally, we can conclude that
\[
E[\mu_0]= E_0,
\]
so $\mu_0$ is a global energy minimizer in $\mathcal{P}_{1, \p}(M)$.
\end{proof}

\begin{remark}
As noted in Remark \ref{R4.2}, as $c_m\to0+$ the exponential function 
\[
\left(\frac{\sinh(\sqrt{c_m}\theta)}{\sqrt{c_m}}\right)^\lambda
\]
converges to the algebraic function $\theta^\lambda$. Therefore, in such limit, the exponential growth condition on $h$ from \eqref{condiH} reduces to an algebraic growth condition, similar to the assumptions used to show the existence of global energy minimizers in the Euclidean case $c_m=0$ (see \cite[Proposition 8]{carrillo2019reverse}). 
\end{remark}

%%%%%%%%%%

\section{Unbounded sectional curvatures}
\label{sect:unbounded}
\setcounter{equation}{0}

In this section we consider more general assumptions on the sectional curvatures of the manifold -- see \eqref{eqn:var-bounds}. In particular, the curvatures are allowed to become unbounded at infinity. One of the main tools in this case is the generalized comparison results in Theorem \ref{lemma:Chavel-thms-gen}. Similar to the constant bounds case, we will treat separately the necessary and sufficient conditions on the interaction potential for global energy minimizers to exist. Throughout the section, $\p$ is an arbitrary  fixed pole on $M$.

\subsection{Nonexistence of a global minimizer}
\label{subsect:nonexist-var}
We assume here that the sectional curvatures of $M$ satisfy
\begin{equation}
\label{eqn:K-upperb-un}
\mathcal{K}(x;\sigma)\leq -c_M(r_x) < 0,
\end{equation}
for all $x \in M$ and all two-dimensional subspaces $\sigma \subset T_x M$, where $c_M(\cdot)$ is a positive continuous function of the distance from the pole $\p$.

We will use again probability density functions in the form \eqref{eqn:rhoR} and look into the behaviours $R\to 0+$ and $R\to \infty$. By the exact calculation of the entropy in \eqref{eqn:entropy-rhoR} and the simple estimate of the interaction energy in \eqref{eqn:intenergy-rhoR}, one reaches \eqref{Eupperbound}, which we list it again for convenience:
\[
E[\rho_R]\leq -\frac{|B_R(\p)|^{1-\m}}{1-\m}+\frac{h(2R)}{2}.
\]

{\em Behaviour at zero.} If $h$ is singular at origin, i.e. $\lim_{\theta\to0+}h(\theta)=-\infty$, then $E$ is unbounded from below, and has no global minimizer. 

{\em Behaviour at infinity.} By the volume lower bound in Theorem \ref{lemma:Chavel-thms-gen} we have
\[
|B_R(\p)|\geq \dm w(\dm)\int_0^R\psi_M(\theta)^{\dm-1}\d\theta,
\]
where $\psi_M$ is the solution to the second IVP in \eqref{eqn:psimM}. This yields
\[
E[\rho_R]\leq -\frac{1}{1-\m}\left(\dm w(\dm)\int_0^R \psi_M(\theta)^{\dm-1}\d\theta\right)^{1-\m}+\frac{h(2R)}{2}.
\]

Hence, if $h$ satisfies
\[
\lim_{R\to\infty}\left(-\frac{1}{1-\m}\left(\dm w(\dm)\int_0^R \psi_M(\theta)^{\dm-1}\d\theta\right)^{1-\m}+\frac{h(2R)}{2}\right)=-\infty,
\]
then there is no global energy minimizer as $\lim_{R\to\infty}E[\rho_R]=-\infty$. Since $\int_0^R\psi_M(\theta)^{\dm-1}\d \theta$ tends to infinity as $R \to \infty$, a sufficient condition for the above inequality can be expressed as
\begin{align}\label{N-1}
\lim_{R\to\infty}\frac{h(R)}{\left(\int_0^{R/2} \psi_M(\theta)^{\dm-1}\d\theta\right)^{1-\m}}=0.
\end{align}
Also, for any $0<\delta<R/2$, we have 
\[
\int_0^{R/2}\psi_M(\theta)^{\dm-1}\d \theta\geq \int_{R/2-\delta}^{R/2}\psi_M(\theta)^{\dm-1}\d \theta\geq  \int_{R/2-\delta}^{R/2}\psi_M(R/2-\delta)^{\dm-1}\d \theta=\delta \psi_M(R/2-\delta)^{\dm-1}.
\]
Hence, a sufficient condition for \eqref{N-1} to hold is
\begin{align}\label{N-2}
\lim_{R\to\infty}\frac{h(R)}{\psi_M(R/2-\delta)^{(\dm-1)(1-\m)}}=0,
\end{align}
for some $\delta>0$.

Assume in addition that $c_M(\cdot)$ satisfies
\begin{equation}
\label{eqn:cM-32}
\lim_{\theta\to\infty}\frac{\overline{D}c_M(\theta)}{c_M(\theta)^{3/2}}=0.
\end{equation}
Then, by Lemma \ref{estpsi}, for any $0<\epsilon<1$, there exists $\theta_0>0$ such that 
\begin{equation}
\label{eqn:psiM-lb}
\psi_M(\theta_0)\exp\left((1-\epsilon)\int_{\theta_0}^\theta \sqrt{c_M(t)}\d t\right)\leq \psi_M(\theta),\qquad\forall \theta\geq\theta_0.
\end{equation}
Using \eqref{eqn:psiM-lb}, we infer that a sufficient condition for \eqref{N-2} is
\[
\lim_{R\to\infty}\frac{h(R)}{\exp\left((1-\epsilon)(\dm-1)(1-\m)\int_0^{R/2-\delta} \sqrt{c_M(t)}\d t)\right)}=0.
\]

We now put together the above considerations in the following theorem.
\begin{theorem}[Nonexistence by blow-up or spreading: unbounded curvature case]
\label{thm:nonexist-un}
Let $M$ be a $\dm$-dimensional Cartan-Hadamard manifold with sectional curvatures that satisfy \eqref{eqn:K-upperb-un}, for a positive continuous function $c_M(\cdot)$ that satisfies \eqref{eqn:cM-32}.  Let  $0<\m<1$ and assume that $h:[0,\infty) \to [-\infty,\infty)$ is a non-decreasing lower semi-continuous function which satisfies either
\[
\lim_{\theta\to0+}h(\theta)=-\infty,
\]
or there exists $\delta>0$ and $0<\epsilon<1$ such that
\begin{align}\label{con-nonexist-un}
\lim_{\theta\to\infty}\frac{h(\theta)}{\exp\left((1-\epsilon)(\dm-1)(1-\m)\int_0^{\theta/2-\delta} \sqrt{c_M(t)}\d t)\right)}=0.
\end{align}
Then the energy functional \eqref{eqn:energy} is unbounded from below, and therefore has no global minimizer in $\calP(M)$.
\end{theorem}

\begin{example}
\label{ex:nonexist}
For certain simple functions $c_M$ the condition \eqref{con-nonexist-un} can be worked out explicitly.
\begin{enumerate}
\item Constant: $c_M(\theta)\equiv c_M$. For this case we essentially recover Theorem \ref{thm:nonexist}. The only difference is the additional coefficient $1-\epsilon$, which can in fact be arbitrarily close to $1$.

\item Power law: $c_M(\theta)=\theta^k$ with $k>0$. Condition \eqref{con-nonexist-un} is equivalent in this case to
\[
\lim_{\theta \to\infty}\frac{h(\theta)}{\exp\left(\frac{(1-\epsilon)(\dm-1)(1-\m)}{k/2+1}\left(\frac{\theta}{2}-\delta\right)^{\frac{k}{2}+1}\right)}=0.
\]
Hence, a necessary condition for global minimizers to exist is that the interaction potential grows exponentially (with a certain exponent on $\theta$ that depends on $k$) at infinity.

\item Exponential growth: $c_M(\theta)=e^{\beta\theta}$ with $\beta>0$.  Condition \eqref{con-nonexist-un} is now equivalent to
\[
\lim_{\theta \to\infty}\frac{h(\theta)}{\exp\left(\frac{(1-\epsilon)(\dm-1)(1-\m)}{\beta/2}\exp\left(\frac{\theta \beta}{4}\right)\right)}=0.
\]
Notably, a very strong growth of the potential (double exponential) is needed to contain the diffusion when the curvatures of the manifold grow exponentially fast at infinity.
\end{enumerate}
\end{example}

%%%%%

\subsection{Existence of a global minimizer}
\label{subsect:exist-var}
In this section, we study the energy minimization problem on $M$ which satisfies
\begin{equation}
\label{eqn:K-lowerbun}
-c_m(r_x) \leq \mathcal{K}(x;\sigma)\leq 0, 
\end{equation}
for all $x \in M$ and all two-dimensional subspaces $\sigma \subset T_x M$,  where $c_m(\cdot)$ is a positive, non-decreasing and continuous function of the distance from the pole $\p$.  Relevant for this section are the upper bounds in Theorem \ref{lemma:Chavel-thms-gen}, with $\psi_m$ the solution to the first IVP in \eqref{eqn:psimM}. 

\begin{lemma}\label{newpsiest}
Let $\psi_m$ be the solution to the first IVP in \eqref{eqn:psimM}, and let $R>0$. Then we have
\[
\psi_m(\theta)\geq \psi_m(R)\exp(\sqrt{c_m(R)}(\theta-R)), \qquad\forall\theta\geq R,
\]
and
\[
\psi_m(\theta)\leq \psi_m(R)\exp(\sqrt{c_m(R)}(\theta-R)),\qquad\forall 0<\theta \leq R.
\]
\end{lemma}

\begin{proof}
The proof follows from some basic ODE considerations. See Appendix \ref{appendix:psimest}.
\end{proof}

The Carlson-Levin inequality in Lemma \ref{mainineq} can be further generalized to manifolds of unbounded curvature. The result is the following.
\begin{lemma}[Generalized Carlson--Levin inequality (unbounded curvature version)]\label{maininequn}
Let $M$ be a $\dm$-dimensional Cartan--Hadamard manifold whose sectional curvatures satisfy \eqref{eqn:K-lowerbun}, for some positive continuous function $c_m(\cdot)$, and let $\p$ be a fixed pole on $M$. Also, let $0<q<1$ and $\lambda>\frac{(\dm-1)(1-q)}{q}$. Then there exists a positive constant $\tilde{C}_1>0$ that depends on $\lambda, q, c_m$, and $\dm$, such that
\[
\int_M\rho(x)^q \dx\leq \tilde{C}_1(\lambda, q, c_m, \dm)\left(\int_M\rho(x) \dx\right)^{(1-p)q}\left(\int_M\psi_m(r_x)^\lambda \rho(x) \dx \right)^{pq},
\]
for any $\rho(\cdot) \geq 0$, where $p=\frac{(\dm-1)(1-q)}{\lambda q}\in(0, 1)$. 
\end{lemma}

\begin{proof}
The arguments follow very closely the proof of Lemma \ref{mainineq}. See Appendix \ref{appendix:gen-CL-ineq}. 
\end{proof}

To show the existence of global minimizers we assume that $h$ satisfies 
\begin{equation}
\label{eqn:condH0-unb}
h(0)=0,
\end{equation}
and
\begin{equation}
\label{condiH-unb0}
%h(0)=0\quad\text{and}\quad 
\liminf_{\theta\to\infty}\frac{h(\theta)}{\exp\left(\tilde{\lambda} \int_0^\theta\sqrt{c_m(\theta)}\d t\right)}>0,\qquad \text{ for some }\tilde{\lambda}>\frac{(\dm-1)(1-\m)}{\m}.
\end{equation}
Condition \eqref{eqn:condH0-unb} simply sets $h(0)$ to a constant (this is the same condition as in \eqref{eqn:condH0}), as $h$ cannot be singular at origin for a ground state to exist. Condition \eqref{condiH-unb0} generalizes \eqref{condiH} from the constant bound case; see also the necessary condition \eqref{con-nonexist-un} from Theorem \ref{thm:nonexist-un}. The reason we use notation $\tilde{\lambda}$ in \eqref{condiH-unb0} is to reserve the symbol $\lambda$ for a different constant, in order to transfer more easily the arguments made in Section \ref{sect:existence} -- see \eqref{condiH-unb} below.

Set $\lambda$ such that $\tilde{\lambda} > \lambda > \frac{(\dm-1)(1-\m)}{\m}$. By Lemma \ref{estpsi} we write
\begin{align*}
\liminf_{\theta\to\infty}\frac{h(\theta)}{\psi_m(\theta)^\lambda}&\geq\liminf_{\theta\to\infty}\frac{h(\theta)}{\theta^\lambda\exp\left(\lambda\int_0^\theta\sqrt{c_m(\theta)}\d t\right)}\\
& = \liminf_{\theta\to\infty}\frac{h(\theta)}{\exp\left(\tilde{\lambda} \int_0^\theta\sqrt{c_m(\theta)}\d t\right)}\times\frac{\exp\left((\tilde{\lambda}-\lambda)\int_0^\theta\sqrt{c_m(\theta)}\right)}{\theta^{\lambda}}\\
&>0,
\end{align*}
where for the last inequality we used \eqref{condiH-unb0} and 
\[
\lim_{\theta\to\infty}\frac{\exp\left((\tilde{\lambda}- \lambda)\int_0^\theta\sqrt{c_m(\theta)}\right)}{\theta^{\tilde{\lambda}}}=\infty.
\]
Hence, if $h$ satisfies \eqref{condiH-unb0}, then it also satisfies
\begin{equation}
\label{condiH-unb}
\liminf_{\theta\to\infty}\frac{h(\theta)}{\psi_m(\theta)^\lambda} >0, \qquad \text{ for some } \lambda >\frac{(\dm-1)(1-\m)}{\m}.
\end{equation}

\begin{remark}
 \label{rmk:weaker-cond}
 \normalfont
 In the proof of Theorem \ref{thm:existence-un} we will use in fact the weaker assumption \eqref{condiH-unb} on the growth at infinity of $h$. We choose however to list the growth assumption as in \eqref{condiH-unb0} as this is given in terms of the known function $c_m(\cdot)$, rather than the unknown solution $\psi_m(\cdot)$ of the IVP.
 \end{remark}

\begin{proposition}[Tightness of the minimizing sequences]
\label{prop:tightnessun}
Let $M$ be a $\dm$-dimensional Cartan-Hadamard manifold whose sectional curvatures satisfy \eqref{eqn:K-lowerbun}, where $c_m(\cdot)$ is a positive, non-decreasing, and continuous function of the distance from a fixed pole $\p$. Also, assume $0<\m<1$ and that $h$ is lower semi-continuous, non-decreasing, and satisfies \eqref{eqn:condH0-unb} and \eqref{condiH-unb0}. Then, any minimizing sequence $\{\mu_n\}_{n \geq 1}\subset \mathcal{P}_{1, \p}(M)$ of the functional $E[\cdot]$ is tight.
\end{proposition}
\begin{proof}[Sketch of the proof]
The proof follows closely the arguments in the proof of Proposition \ref{prop:tightness} by replacing $ \sinh(\sqrt{c_m}\theta)/\sqrt{c_m}$  with $\psi_m(\theta)$. An analogue of \eqref{ineq:Gmun} can first be derived by using Lemma \ref{maininequn}, where $\lambda$ is the constant from \eqref{condiH-unb}. Then, instead of \eqref{eqn:hosinh} one can use \eqref{condiH-unb} to find an analogue of \eqref{h-sinhest} with $\psi_m(\theta)$ in place of $\sinh(\sqrt{c_m}\theta)/\sqrt{c_m}$. From there on, the arguments in the proof of Proposition \ref{prop:tightness} follow exactly, and a uniform upper bound as in \eqref{eqn:C3ineq} can be derived, i.e.,
\begin{align}\label{eqn:C3inequn}
\tilde{C}_3\geq \int_M\psi_m(r_x)^\lambda \d\mu_n(x), \qquad\forall n\in\mathbb{N},
\end{align}
for some constant $\tilde{C}_3>0$ which does not depend on $\mu_n$. Eventually, \eqref{C3inq2} is adjusted into
\begin{equation}
\label{C3inq2-un}
\int_{r_x\geq R}\d\mu_n(x)\leq \frac{\tilde{C}_3}{\psi_m(R)^\lambda}, \qquad \forall n \in \mathbb{N},
\end{equation}
where the constant $\tilde{C}_3$ does not depend on $R$ or $n$. From here, we infer that $\{\mu_n\}_{n\geq1}$ is tight.
\end{proof}

\begin{proposition}[Lower semi-continuity of the energy]\label{lscEnergyun}
Let $M$ and $h$ satisfy the assumptions of Proposition \ref{prop:tightnessun}, and $\p$ be a fixed pole in $M$. Then, for $0<\m<1$, $E[\cdot]$ is lower semi-continuous on minimizing sequences, i.e., for any minimizing sequence $\{\mu_n\}_{n \geq 1}\subset\mathcal{P}_{1, \p}(M)$ of $E$ that converges weakly to $\mu_0$, we have
\[
\liminf_{n\to\infty}E[\mu_n]\geq E[\mu_0].
\]
\end{proposition}

\begin{proof}[Sketch of the proof]
The proof is similar to that of Proposition \ref{lscEnergy}, by replacing $\frac{\sinh(\sqrt{c_m}\theta)}{\sqrt{c_m}}$  with $\psi_m(\theta)$.  Indeed, denote by $\rho_n$ and $\rho_0$ the absolutely continuous parts of $\mu_n$ and $\mu_0$, respectively. Following the argument leading to \eqref{entext}, by using the generalized Carlson-Levin inequality in Lemma \ref{maininequn} (see also \eqref{eqn:C3inequn} and \eqref{C3inq2-un}), one finds
\[
\int_{B_r(\p)^c}\rho_n(x)^\m \dx \leq \frac{\tilde{C}_1\tilde{C}_3^\m}{\psi_m(r)^{\lambda(1-p)\m}}.
\]
Here, $r>0$ is a fixed radius, $\lambda$ is the constant from \eqref{condiH-unb}, $p \in (0,1)$ is defined in Lemma \ref{maininequn}, and $\tilde{C}_1$ and $\tilde{C}_3$ are constants that do not depend on $r$ or $n$.

Then, since $\lim_{r\to\infty}\psi_m(r)=\infty$, one can follow the argument in the proof of Proposition \ref{lscEnergy} to infer the lower semi-continuity of the entropy:
\[
\frac{1}{\m-1} \int_M\rho_0(x)^\m\d x\leq \liminf_{n\to\infty}\left(\frac{1}{\m-1} \int_M\rho_n(x)^\m\d x\right).
\]
The lower semi-continuity of the interaction energy can be obtained from the same argument as in Part 2 of the proof of Proposition \ref{lscEnergy}. The thesis of the theorem follows from combining the two facts.
\end{proof}

The main result is the following theorem.
\begin{theorem}[Existence of global minimizer: unbounded curvature case]
\label{thm:existence-un}
Let $M$ be a $\dm$-dimensional Cartan-Hadamard manifold whose sectional curvatures that satisfy \eqref{eqn:K-lowerbun}, where $c_m(\cdot)$ is a positive, non-decreasing, and continuous function of the distance from a fixed pole $\p$. Also, assume $0<\m<1$ and that $h$ is lower semi-continuous, non-decreasing, and satisfies  \eqref{eqn:condH0-unb} and \eqref{condiH-unb0}. Then, the energy functional $E[\cdot]$ given by \eqref{eqn:energy} has a global minimizer in $\calP_{1,\p}(M)$.
\end{theorem}
\begin{proof}
The proof follows the same argument used to prove Theorem \ref{thm:existence}, using the tightness of minimizing sequences (Proposition \ref{prop:tightnessun}), the lower semi-continuity of the energy functional (Proposition \ref{lscEnergyun}), and the conservation of the centre of mass (Lemma \ref{lemma:CM}). We also note that a bound analogous to \eqref{eqn:E-lb} (see Remark \ref{rmk:E-lb}) holds in the present case as well, following the proof of Proposition \ref{prop:tightnessun}. \end{proof}

\begin{example} 
\label{ex:exist}
We consider here several examples of functions $c_m$ -- see also Example \ref{ex:nonexist}.
\begin{enumerate}
\item Constant: $c_m(\theta)\equiv c_m$. For this case, condition \eqref{condiH-unb0} reduces to \eqref{condiH}, so we recover Theorem \ref{thm:existence}.
\item Power law: $c_m(\theta)=\theta^k$ with $k>0$. Condition \eqref{condiH-unb0} is equivalent in this case to
\[
\liminf_{\theta\to\infty}\frac{h(\theta)}{\exp\left(\frac{\lambda}{k/2+1}\theta^{k/2+1}\right)}>0.
\]
\item Exponential growth: $c_m(\theta)=e^{\beta\theta}$ with $\beta>0$.  In this case, \eqref{condiH-unb0} becomes
\[
\liminf_{\theta\to\infty}\frac{h(\theta)}{\exp\left(\frac{\lambda}{\beta}\exp\left(\beta\theta\right)\right)}>0.
\]
Therefore, an exponential and a double exponential growth of the potential contains the spreading when the curvatures of the manifold grow algebraically, respectively exponentially, at infinity.
\end{enumerate}
\end{example}

\begin{remark}
Two two conditions \eqref{con-nonexist-un} and \eqref{condiH-unb0} are complementary, in the sense that they cannot hold at the same time. Indeed, suppose the sectional curvatures of $M$ satisfy both a lower and an upper bound: $-c_m(r_x) \leq \mathcal{K}(x;\sigma)\leq -c_M(r_x)$, for all points $x$ and all sections $\sigma$, where $c_m(\cdot)$ is non-decreasing. Suppose $h$ satisfies \eqref{condiH-unb0} for some $\tilde{\lambda}>\frac{(\dm-1)(1-\m)}{\m}$. Then, since $0<\m<1$, we have
\[
\tilde{\lambda}>\frac{(\dm-1)(1-\m)}{\m}>(1-\epsilon)(\dm-1)(1-\m),
\]
for any $0<\epsilon<1$.

Therefore, with $\delta>0$, we can estimate
\begin{align*}
\frac{h(\theta)}{\exp\left((1-\epsilon)(\dm-1)(1-\m)\int_0^{\theta/2-\delta} \sqrt{c_M(t)}\d t)\right)} & \geq \frac{h(\theta)}{\exp\left(\tilde{\lambda}\int_0^{\theta/2-\delta} \sqrt{c_M(t)}\d t)\right)}\\
& \geq \frac{h(\theta)}{\exp\left(\tilde{\lambda}\int_0^{\theta} \sqrt{c_M(t)}\d t)\right)} \\
& \geq \frac{h(\theta)}{\exp\left(\tilde{\lambda}\int_0^{\theta} \sqrt{c_m(t)}\d t)\right)}
\end{align*}
By sending $\theta \to \infty$ we then find that \eqref{con-nonexist-un} cannot hold simultaneously with \eqref{condiH-unb0}, as otherwise it would lead to the following contradiction:
\begin{align*}
0=&\lim_{\theta\to\infty}\frac{h(\theta)}{\exp\left((1-\epsilon)(\dm-1)(1-\m)\int_0^{\theta/2-\delta} \sqrt{c_M(t)}\d t)\right)}\\
\geq&\liminf_{\theta\to\infty}\frac{h(\theta)}{\exp\left(\tilde{\lambda}\int_0^{\theta} \sqrt{c_m(t)}\d t)\right)}>0.
\end{align*}
\end{remark}

%%%%%%%%%%

\appendix
\section{Proof of Lemma \ref{newpsiest}}
\label{appendix:psimest}

Define $f:[0, \infty) \to \bbr$ by
\[
f(\theta)=\frac{\psi_m(\theta)}{\psi_m'(\theta)}.
\]
Then, using  \eqref{eqn:psimM}, we find that $f$ satisfies
\[
\begin{cases}
f'(\theta) = 1-c_m(\theta)f^2(\theta), \qquad\forall \theta>0,\\[2pt]
%=\frac{\psi_m'(\theta)^2-\psi_m(\theta)\psi_m(\theta)''}{\psi'_m(\theta)^2}=
f(0)=0,
\end{cases}
\]
and since $c_m$ is non-decreasing, we further have
\[
f'(\theta)\geq 1-c_m(R)f(\theta)^2, \qquad\forall 0 < \theta\leq R,
\]
and
\[
f'(\theta)\leq 1-c_m(R)f(\theta)^2, \qquad\forall R\leq \theta.
\]

Let $g:[0, \infty)\to\bbr$ be the solution to the following IVP:
\[
\begin{cases}
g'(\theta)=1-c_m(R)g(\theta)^2, \qquad \forall \theta>0,\\[2pt]
g(R)=f(R).
\end{cases}
\]
By the comparison theorem in ODEs, we first note that
\begin{align}\label{compODE}
f(\theta)\leq g(\theta),\qquad  \forall \theta > 0.
\end{align}
The IVP for $g$ has the explicit solution
\[
g(\theta)=\frac{1}{\sqrt{c_m(R)}}\tanh\left(\mathrm{arctanh}(\sqrt{c_m(R)}f(R))+\sqrt{c_m(R)}(\theta-R)\right),
\]
which used in \eqref{compODE}, gives
\[
f(\theta)\leq\frac{1}{\sqrt{c_m(R)}}\tanh\left(\mathrm{arctanh}(\sqrt{c_m(R)}f(R))+\sqrt{c_m(R)}(\theta-R)\right), \qquad\forall \theta > 0.
\]
From the definition of $f$, we further get
\[
\sqrt{c_m(R)}\coth\left(\mathrm{arctanh}(\sqrt{c_m(R)}f(R))+\sqrt{c_m(R)}(\theta-R)\right)\leq \frac{\psi_m'(\theta)}{\psi_m(\theta)},\qquad\forall \theta>0,
\]
which is equivalent to
\begin{align}\label{psidiffineq}
\frac{\d}{\d\theta}\ln\left(\sinh\left(\mathrm{arctanh}(\sqrt{c_m(R)}f(R))+\sqrt{c_m(R)}(\theta-R)\right)\right)\leq \frac{\d}{\d\theta}\ln\psi_m(\theta).
\end{align}

Now, we estimate $\psi(\theta)$ for $\theta > R$ and $0<\theta < R$ separately.
\medskip

\noindent(Case 1: $\theta > R$) Integrate \eqref{psidiffineq} from $R$ to $\theta>R$ to find
\[
\psi_m(\theta)\geq \psi_m(R)\left(\frac{\sinh\left(\mathrm{arctanh}(\sqrt{c_m(R)}f(R))+\sqrt{c_m(R)}(\theta-R)\right)}{\sinh\left(\mathrm{arctanh}(\sqrt{c_m(R)}f(R))\right)}\right), \qquad \forall \theta>R.
\]
If we set
\begin{equation}
\label{eqn:calB}
\mathcal{B}=\mathrm{arctanh}(\sqrt{c_m(R)}f(R)),
\end{equation}
then the right hand side of the above inequality can be estimated as follows:
\begin{align*}
&\psi_m(R)\left(\frac{\sinh(\mathcal{B}+\sqrt{c_m(R)}(\theta-R)}{\sinh\mathcal{B}}\right)\\
%&=\psi_m(R)\left(\frac{\exp(\mathcal{B}+\sqrt{c_m(R)}(\theta-R))-\exp(-\mathcal{B}-\sqrt{c_m(R)}(\theta-R))}{\exp(\mathcal{B})-\exp(-\mathcal{B})}\right)\\
&=\psi_m(R)\exp(\sqrt{c_m(R)}(\theta-R))\left(\frac{1-\exp(-2\mathcal{B}-2\sqrt{c_m(R)}(\theta-R))}{1-\exp(-2\mathcal{B})}\right)\\
%&\geq\psi_m(R)\exp(\sqrt{c_m(R)}(\theta-R))\left(\frac{1-\exp(-2\mathcal{B}-2\sqrt{c_m(R)}(R-R))}{1-\exp(-2\mathcal{B})}\right)\\
& \geq \psi_m(R)\exp(\sqrt{c_m(R)}(\theta-R)).
\end{align*}
This leads to the first inequality we need to show (for $\theta \geq R$).
\medskip

\noindent(Case 2: $0<\theta < R$) Integrate \eqref{psidiffineq} from $\theta< R$ to $R$ to get
\[
\psi_m(\theta)\leq  \psi_m(R)\left(\frac{\sinh\left(\mathrm{arctanh}(\sqrt{c_m(R)}f(R))+\sqrt{c_m(R)}(\theta-R)\right)}{\sinh\left(\mathrm{arctanh}(\sqrt{c_m(R)}f(R))\right)}\right), \qquad \forall 0<\theta <R.
\]
Again, use notation \eqref{eqn:calB} to simplify and estimate the right hand side above as
\begin{align*}
&\psi_m(R)\left(\frac{\sinh(\mathcal{B}+\sqrt{c_m(R)}(\theta-R)}{\sinh\mathcal{B}}\right)\\
%&=\psi_m(R)\left(\frac{\exp(\mathcal{B}+\sqrt{c_m(R)}(\theta-R))-\exp(-\mathcal{B}-\sqrt{c_m(R)}(\theta-R))}{\exp(\mathcal{B})-\exp(-\mathcal{B})}\right)\\
&=\psi_m(R)\exp(\sqrt{c_m(R)}(\theta-R))\left(\frac{1-\exp(-2\mathcal{B}-2\sqrt{c_m(R)}(\theta-R))}{1-\exp(-2\mathcal{B})}\right)\\
%&\leq\psi_m(R)\exp(\sqrt{c_m(R)}(\theta-R))\left(\frac{1-\exp(-2\mathcal{B}-2\sqrt{c_m(R)}(R-R))}{1-\exp(-2\mathcal{B})}\right)\\
&\leq \psi_m(R)\exp(\sqrt{c_m(R)}(\theta-R)),
\end{align*}
leading to the second inequality to be shown (for $0<\theta<R$).

%%%%%

\section{Proof of Lemma \ref{maininequn}}
\label{appendix:gen-CL-ineq}

Fix a nonnegative function $\rho$ on $M$ and fix $R>0$. By H\"{o}lder inequality with exponents $\frac{1}{q}$ and $\frac{1}{1-q}$, we get
\begin{equation}
\label{eqn:est-tlrun}
\begin{aligned}
\int_{r_x<R}\rho(x)^q \dx&\leq \left(\int_{r_x<R}\rho(x) \dx\right)^q\left(\int_{r_x< R}1\, \dx\right)^{1-q}\\[5pt]
&\leq\left(\int_M\rho(x) \dx\right)^q\left(\dm w(\dm)\int_0^R\psi_m(\theta)^{\dm-1}\d\theta\right)^{1-q},
\end{aligned}
\end{equation}
where in the second inequality we used the volume upper bound in Theorem \ref{lemma:Chavel-thms-gen}. Also by H\"{o}lder inequality with exponents $\frac{1}{q}$ and $\frac{1}{1-q}$, we find
\begin{equation}
\label{eqn:est-tgrun}
\begin{aligned}
\int_{r_x\geq R}\rho(x)^q \dx&\leq \left(\int_{r_x\geq R}
\psi_m(r_x)^\lambda\rho(x) \dx
\right)^q
\left(
\int_{r_x\geq R}\psi_m(r_x)^{-\frac{\lambda q}{1-q}} \dx
\right)^{1-q}\\[5pt]
&\leq  \left(\int_M
\psi_m(r_x)^\lambda\rho(x) \dx
\right)^q
\left(
\dm w(\dm)\int_R^\infty \psi_m(\theta)^{\dm-1-\frac{\lambda q}{1-q}}\d\theta\right)^{1-q},
\end{aligned}
\end{equation}
where for the second inequality we used a very similar argument to that used to derive the second inequality in \eqref{eqn:est-tgr}. Namely, we changed variables to geodesic spherical coordinates $x=\exp_\p(\theta u)$, with $\theta>0$ and $u \in \bbs^{\dm-1}$, in the second integral, and used the upper bound on $J(\exp_\p)$ from Theorem \ref{lemma:Chavel-thms-gen} to bound from above $\d V(x)$ -- see also \eqref{eqn:dV-bound}.

Denote
\begin{align*}
&C(R):=\left(\dm w(\dm)\int_0^R\psi_m(\theta)^{\dm-1}\d\theta\right)^{1-q},\\[2pt]
&D(R):=\left(
\dm w(\dm)\int_R^\infty \psi_m(\theta)^{\dm-1-\frac{\lambda q}{1-q}}\d\theta\right)^{1-q}.
\end{align*}
Then, by combining \eqref{eqn:est-tlrun} and \eqref{eqn:est-tgrun} we get
\begin{equation}
\label{eqn:est-rhoqun}
\int_M\rho(x)^q \dx\leq C(R)\left(\int_M
\rho(x) \dx
\right)^q+D(R)\left(\int_M
\psi_m(r_x)^\lambda\rho(x) \dx
\right)^q.
\end{equation}
From Lemma \ref{newpsiest}, we have
\begin{align*}
C(R)&\leq \left(\dm w(\dm)\int_0^R \left(\psi_m(R)\exp(\sqrt{c_m(R)}(\theta-R) \right)^{\dm-1}\d\theta\right)^{1-q}\\
%&=\left(\dm w(\dm)\psi_m(R)^{\dm-1} \exp(-(\dm-1)\sqrt{c_m(R)}R)\int_0^R\exp((\dm-1)\sqrt{c_m(R)}\theta)\d\theta\right)^{1-q}\\
&=\left(\frac{\dm w(\dm)\psi_m(R)^{\dm-1} \exp(-(\dm-1)\sqrt{c_m(R)}R)}{(\dm-1)\sqrt{c_m(R)}}\left(\exp((\dm-1)\sqrt{c_m(R)}R)-1\right)\right)^{1-q}\\
&<\left(\frac{\dm w(\dm)\psi_m(R)^{\dm-1} \exp(-(\dm-1)\sqrt{c_m(R)}R)}{(\dm-1)\sqrt{c_m(R)}}\exp((\dm-1)\sqrt{c_m(R)}R)\right)^{1-q}\\
&=\left(\frac{\dm w(\dm) }{(\dm-1)\sqrt{c_m(R)}}\right)^{1-q}\psi_m(R)^{(\dm-1)(1-q)}\\
&\leq\left(\frac{\dm w(\dm) }{(\dm-1)\sqrt{c_m(0)}}\right)^{1-q}\psi_m(R)^{(\dm-1)(1-q)}
\end{align*}
and (note that $d-1-\frac{\lambda q}{1-q}<0$)
\begin{align*}
D(R) &\leq \left(\dm w(\dm)\int_R^\infty (\psi_m(R)\exp(\sqrt{c_m(R)}(\theta-R))^{\dm-1-\frac{\lambda q}{1-q}}\d\theta\right)^{1-q}\\
%&= \left(\dm w(\dm)\psi_m(R)^{\dm-1-\frac{\lambda q}{1-q}}\exp\left(\left(-\dm+1+\frac{\lambda q}{1-q}\right)\sqrt{c_m(R)}R\right)\int_R^\infty \exp\left(\left(\dm-1-\frac{\lambda q}{1-q}\right)\sqrt{c_m(R)}\theta\right)\d\theta\right)^{1-q}\\
&= \left(\frac{\dm w(\dm)\psi_m(R)^{\dm-1-\frac{\lambda q}{1-q}}\exp\left(\left(-\dm+1+\frac{\lambda q}{1-q}\right)\sqrt{c_m(R)}R\right)}{\left(-\dm+1+\frac{\lambda q}{1-q}\right)\sqrt{c_m(R)}} \exp\left(\left(\dm-1-\frac{\lambda q}{1-q}\right)\sqrt{c_m(R)}R\right)\right)^{1-q}\\
&= \left(\frac{\dm w(\dm)}{\left(-\dm+1+\frac{\lambda q}{1-q}\right)\sqrt{c_m(R)}} \right)^{1-q}\psi_m(R)^{(\dm-1)(1-q)-\lambda q}\\
&\leq \left(\frac{\dm w(\dm)}{\left(-\dm+1+\frac{\lambda q}{1-q}\right)\sqrt{c_m(0)}} \right)^{1-q}\psi_m(R)^{(\dm-1)(1-q)-\lambda q}.
\end{align*}

If we set
\begin{align*}
&\tilde{\alpha}_1=\left(\frac{\dm w(\dm) }{(\dm-1)\sqrt{c_m(0)}}\right)^{1-q},\qquad \beta_1=(\dm-1)(1-q),\\
&\tilde{\alpha}_2=\left(\frac{\dm w(\dm)}{\left(-\dm+1+\frac{\lambda q}{1-q}\right)\sqrt{c_m(0)}} \right)^{1-q},\qquad \beta_2=\lambda q-(\dm-1)(1-q),
\end{align*}
then by the assumptions on $q$ and $\lambda$ all these constants are positive and we get
\[
C(R)\leq \alpha_1\psi_m(R)^{\beta_1},\quad D(R)\leq \alpha_2\psi_m(R)^{-\beta_2}.
\]
Finally, we substitute the above result to \eqref{eqn:est-rhoqun} to get
\begin{align}\label{eqn:est-rhoqun-2}
\int_M\rho(x)^q \dx \leq \tilde{\alpha}_1\psi_m(R)^{\beta_1}\left(\int_M
\rho(x) \dx
\right)^q+\tilde{\alpha}_2\psi_m(R)^{-\beta_2}\left(\int_M
\psi_m(r_x)^\lambda\rho(x) \dx
\right)^q.
\end{align}

To find an optimal $R$, we calculate the derivative of the right hand side of \eqref{eqn:est-rhoqun} with respect to $R$, and set it to zero:
\[
0=\tilde{\alpha}_1\beta_1\psi_m(R)^{\beta_1-1}\psi_m'(R)\left(\int_M\rho(x) \dx\right)^q-\tilde{\alpha}_2\beta_2\psi_m(R)^{-\beta_2-1}\psi_m'(R)\left(\int_M\psi_m(r_x)^\lambda\rho(x) \dx\right)^q.
\]
Then, we get
\[
\psi_m(R)^{\beta_1+\beta_2}=\frac{\tilde{\alpha}_2\beta_2}{\tilde{\alpha}_1\beta_1}\frac{\left(\int_M\psi_m(r_x)^\lambda\rho(x)\dx\right)^q}{\left(\int_M\rho(x) \dx\right)^q}.
\]
Since $\psi_m(0)=0$, $\lim_{\theta\to\infty}\psi_m(\theta)=\infty$ and $\psi_m$ is an increasing function, we know there exists a unique $R=R_*>0$ which satisfies the above equation. Finally, we substitute $R=R_*$ into \eqref{eqn:est-rhoqun-2} to find
\begin{align*}
\int_M\rho(x)^q \dx&\leq \tilde{\alpha}_1\psi_m(R_*)^{\beta_1}\left(\int_M
\rho(x) \dx \right)^q+\tilde{\alpha}_2\psi_m(R_*)^{-\beta_2}\left(\int_M \psi_m(r_x)^\lambda\rho(x) \dx \right)^q\\[2pt]
&=\left(
\tilde{\alpha}_1\left(\frac{\alpha_2\beta_2}{\tilde{\alpha}_1\beta_1}\right)^{\frac{\beta_1}{\beta_1+\beta_2}}+\tilde{\alpha}_2\left(\frac{\tilde{\alpha}_2\beta_2}{\tilde{\alpha}_1\beta_1}\right)^{-\frac{\beta_2}{\beta_1+\beta_2}}
\right) \\
& \quad \times \left(\int_M\rho(x) \dx\right)^{\frac{\beta_2 q}{\beta_1+\beta_2}}\left(\int_M\psi_m(r_x)^\lambda\rho(x) \dx\right)^{\frac{\beta_1 q}{\beta_1+\beta_2}}\\[3pt]
&=\tilde{C}_1(\lambda, q, c_m, \dm) \left(\int_M\rho(x) \dx\right)^{(1-p)q}\left(\int_M\psi_m(r_x)^\lambda\rho(x) \dx\right)^{pq},
\end{align*}
where 
\[
\tilde{C}_1(\lambda, q, c_m, \dm) = \left(\tilde{\alpha}_1^{\beta_2} \tilde{\alpha}_2^{\beta_1}\right)^{\frac{1}{\beta_1+\beta_2}}
\left(
\left(\frac{\beta_2}{\beta_1}\right)^{\frac{\beta_1}{\beta_1+\beta_2}} +
\left(\frac{\beta_1}{\beta_2}\right)^{\frac{\beta_2}{\beta_1+\beta_2}}
\right),
\]
and 
$p=\frac{\beta_1}{\beta_1+\beta_2}=\frac{(\dm-1)(1-q)}{\lambda q}$.\\

\noindent\textbf{Acknowledgments}. JAC was partially supported by the Advanced Grant Nonlocal-CPD
(Nonlocal PDEs for Complex Particle Dynamics: Phase Transitions, Patterns and Synchronization) of the European Research Council Executive Agency (ERC) under the European
Union Horizon 2020 research and innovation programme (grant agreement No. 883363).
JAC acknowledge support by the EPSRC grant EP/V051121/1. JAC was also partially
supported by the “Maria de Maeztu” Excellence Unit IMAG, reference CEX2020-001105-M,
funded by MCIN/AEI/10.13039/501100011033/. RF was supported by NSERC Discovery
Grant PIN-341834 during this research. RF also acknowledges an NSERC Alliance International Catalyst grant, which supported H. Park’s visit to University of Oxford, where the
research presented in this paper was initiated.

%%%%%

\bibliographystyle{abbrv}
\def\url#1{}
\bibliography{lit-H.bib}

\def\cprime{$'$}
\begin{thebibliography}{10}

\bibitem{alias2019maximum}
L.~Al{\'\i}as, P.~Mastrolia, and M.~Rigoli.
\newblock {\em Maximum Principles and Geometric Applications}.
\newblock Springer Monographs in Mathematics. Springer International
  Publishing, 2019.

\bibitem{AGS2005}
L.~Ambrosio, N.~Gigli, and G.~Savar{\'e}.
\newblock {\em Gradient flows in metric spaces and in the space of probability
  measures}.
\newblock Lectures in Mathematics ETH Z\"urich. Birkh\"auser Verlag, Basel,
  2005.

\bibitem{Bedrossian11}
J.~Bedrossian.
\newblock Global minimizers for free energies of subcritical aggregation
  equations with degenerate diffusion.
\newblock {\em Appl. Math. Lett.}, 24(11):1927--1932, 2011.

\bibitem{CalvezCarrilloHoffmann2017}
V.~Calvez, J.~Carrillo, and F.~Hoffmann.
\newblock Equilibria of homogeneous functionals in the fair-competition regime.
\newblock {\em Nonlinear Analysis}, 159:85--128, 2017.

\bibitem{CarrilloCraigYao2019}
J.~A. Carrillo, K.~Craig, and Y.~Yao.
\newblock Aggregation-diffusion equations: {D}ynamics, asymptotics, and
  singular limits.
\newblock In N.~Bellomo, P.~Degond, and E.~Tadmor, editors, {\em Active
  Particles, Volume 2: Advances in Theory, Models, and Applications}, Modeling
  and Simulation in Science, Engineering and Technology, pages 65--108.
  Birkh\"auser, 2019.

\bibitem{CaDePa2019}
J.~A. Carrillo, M.~Delgadino, and F.~S. Patacchini.
\newblock Existence of ground states for aggregation-diffusion equations.
\newblock {\em Analysis and Applications}, 17(3):393--423, 2019.

\bibitem{carrillo2019reverse}
J.~A. Carrillo, M.~G. Delgadino, J.~Dolbeault, R.~L. Frank, and F.~Hoffmann.
\newblock Reverse {H}ardy-{L}ittlewood-{S}obolev inequalities.
\newblock {\em Journal de Math{\'e}matiques Pures et Appliqu{\'e}es},
  132:133--165, 2019.

\bibitem{CaDeFrLe2022}
J.~A. Carrillo, M.~G. Delgadino, R.~L. Frank, and M.~Lewin.
\newblock Fast diffusion leads to partial mass concentration in keller–segel
  type stationary solutions.
\newblock {\em Math. Models Methods Appl. Sci.}, 32(4):831--850, 2022.

\bibitem{fernandez2023partial}
J.~A. Carrillo, A.~Fern\'andez-Jim\'enez, and D.~G\'omez-Castro.
\newblock Partial mass concentration for fast-diffusions with non-local
  aggregation terms.
\newblock {\em J. Differential Equations}, 409:700--773, 2024.

\bibitem{CaFePa2024a}
J.~A. Carrillo, R.~C. Fetecau, and H.~Park.
\newblock Existence of ground states for free energies on the hyperbolic space.
\newblock {\em arXiv preprint http://arxiv.org/abs/...}, 2024.

\bibitem{CGV22}
J.~A. Carrillo, D.~G\'omez-Castro, and J.~L. V\'azquez.
\newblock Infinite-time concentration in aggregation-diffusion equations with a
  given potential.
\newblock {\em J. Math. Pures Appl. (9)}, 157:346--398, 2022.

\bibitem{CaHiVoYa2019}
J.~A. Carrillo, S.~Hittmeir, B.~Volzone, and Y.~Yao.
\newblock Nonlinear aggregation-diffusion equations: radial symmetry and long
  time asymptotics.
\newblock {\em Invent. Math.}, 218:889–977, 2019.

\bibitem{CaHoMaVo2018}
J.~A. Carrillo, F.~Hoffmann, E.~Mainini, and B.~Volzone.
\newblock Ground states in the diffusion-dominated regime.
\newblock {\em Calc. Var. Partial Differ. Equ.}, 57:127, 2018.

\bibitem{Chavel2006}
I.~Chavel.
\newblock {\em Riemannian Geometry : A Modern Introduction}.
\newblock Cambridge Studies in Advanced Mathematics. Cambridge University
  Press, second edition, 2006.

\bibitem{cheeger1975comparison}
J.~Cheeger and D.~G. Ebin.
\newblock {\em Comparison theorems in {R}iemannian geometry}, volume~9 of {\em
  North-Holland Mathematical Library}.
\newblock North-Holland Publishing Company, Amsterdam, 1975.

\bibitem{DDFM20}
P.~Degond, A.~Diez, A.~Frouvelle, and S.~Merino-Aceituno.
\newblock Phase transitions and macroscopic limits in a {BGK} model of
  body-attitude coordination.
\newblock {\em J. Nonlinear Sci.}, 30(6):2671--2736, 2020.

\bibitem{DelgadinoXukaiYao2022}
M.~G. Delgadino, X.~Yan, and Y.~Yao.
\newblock Uniqueness and nonuniqueness of steady states of
  aggregation-diffusion equations.
\newblock {\em Comm. Pure Appl. Math.}, 75(1):3--59, 2022.

\bibitem{doCarmo1992}
M.~P. do~Carmo.
\newblock {\em Riemannian Geometry}.
\newblock Mathematics: Theory and Applications. Birkh\"auser, Boston, second
  edition, 1992.

\bibitem{fatkullin2005critical}
I.~Fatkullin and V.~Slastikov.
\newblock Critical points of the onsager functional on a sphere.
\newblock {\em Nonlinearity}, 18(6):2565, 2005.

\bibitem{FeHaPa2021}
R.~C. Fetecau, S.-Y. Ha, and H.~Park.
\newblock An intrinsic aggregation model on the special orthogonal group
  {SO}(3): well-posedness and collective behaviours.
\newblock {\em J. Nonlinear Sci.}, 31(5):74, 2021.

\bibitem{FePa2023b}
R.~C. Fetecau and H.~Park.
\newblock Long-time behaviour of interaction models on {R}iemannian manifolds
  with bounded curvature.
\newblock {\em J. Geom. Anal.}, 33(7):218, 2023.

\bibitem{FePa2024b}
R.~C. Fetecau and H.~Park.
\newblock Aggregation-diffusion energies on {C}artan--{H}adamard manifolds of
  unbounded curvature.
\newblock {\em J. Geom. Anal.}, 34(12):356, 2024.

\bibitem{FePa2024a}
R.~C. Fetecau and H.~Park.
\newblock Ground states for aggregation--diffusion models on
  {C}artan--{H}adamard manifolds.
\newblock {\em J. Lond. Math. Soc.}, 111(2):e70079, 2025.

\bibitem{FePa2021}
R.~C. Fetecau and F.~S. Patacchini.
\newblock Well-posedness of an interaction model on {R}iemannian manifolds.
\newblock {\em Comm. Pure Appl. Anal.}, 21(11):3559--3585, 2021.

\bibitem{FeZh2019}
R.~C. Fetecau and B.~Zhang.
\newblock Self-organization on {R}iemannian manifolds.
\newblock {\em J. Geom. Mech.}, 11(3):397–426, 2019.

\bibitem{GPY17}
J.~Garnier, G.~Papanicolaou, and T.-W. Yang.
\newblock Consensus convergence with stochastic effects.
\newblock {\em Vietnam J. Math.}, 45(1-2):51--75, 2017.

\bibitem{GPY19}
J.~Garnier, G.~Papanicolaou, and T.-W. Yang.
\newblock Mean field model for collective motion bistability.
\newblock {\em Discrete Contin. Dyn. Syst. Ser. B}, 24(2):851--879, 2019.

\bibitem{GrilloMeglioliPunzo2021a}
G.~Grillo, G.~Meglioli, and F.~Punzo.
\newblock Global existence of solutions and smoothing effects for classes of
  reaction–diffusion equations on manifolds.
\newblock {\em J. Evol. Equ.}, 21(2):2339–2375, 2021.

\bibitem{GrilloMeglioliPunzo2021b}
G.~Grillo, G.~Meglioli, and F.~Punzo.
\newblock Smoothing effects and infinite time blowup for reaction-diffusion
  equations: {A}n approach via {S}obolev and {P}oincar\'e inequalities.
\newblock {\em J. Math Pures Appl.}, 151:99--131, 2021.

\bibitem{GrilloMuratoriVazquez2017}
G.~Grillo, M.~Muratori, and J.~L. V\'{a}zquez.
\newblock The porous medium equation on {R}iemannian manifolds with negative
  curvature. {T}he large-time behaviour.
\newblock {\em Advances in Mathematics}, 314:328--377, 2017.

\bibitem{GrilloMuratoriVazquez2019}
G.~Grillo, M.~Muratori, and J.~L. V\'{a}zquez.
\newblock The porous medium equation on {R}iemannian manifolds with negative
  curvature: the superquadratic case.
\newblock {\em Mathematische Annalen}, 373:119–153, 2019.

\bibitem{Helgason2001}
S.~Helgason.
\newblock {\em Differential {G}eometry, {L}ie {G}roups, and {S}ymmetric
  {S}paces}, volume~34 of {\em Graduate Studies in Mathematics}.
\newblock American Mathematical Society, Providence, RI, 2001.

\bibitem{Kaib17}
G.~Kaib.
\newblock Stationary states of an aggregation equation with degenerate
  diffusion and bounded attractive potential.
\newblock {\em SIAM J. Math. Anal.}, 49(1):272--296, 2017.

\bibitem{Santambrogio2015}
F.~Santambrogio.
\newblock {\em Optimal {T}ransport for {A}pplied {M}athematicians}, volume~87
  of {\em Progress in Nonlinear Differential Equations and Their Applications}.
\newblock Birkh\"auser, 2015.

\end{thebibliography}

\end{document}